\documentclass[a4paper, 11pt]{amsart}
\usepackage[margin=1in]{geometry}
\usepackage[utf8]{inputenc}
\usepackage[T1]{fontenc}
\usepackage{amssymb, amsmath, amsfonts}
\usepackage{mathtools}
\usepackage{enumitem}
\usepackage[protrusion=true,expansion=false]{microtype}
\usepackage[hidelinks]{hyperref}
\hypersetup{
  pdftitle={Variational Nonlinearities and Wave Selection at an O(2)-Hopf Bifurcation},
  pdfauthor={Taylan Şengül}
}

\newcommand{\abs}[1]{\lvert #1 \rvert}
\renewcommand{\Re}{\operatorname{Re}}
\renewcommand{\Im}{\operatorname{Im}}
\renewcommand{\det}{\operatorname{det}}
\newcommand{\tr}{\operatorname{tr}}
\newcommand{\zetaR}{\zeta_{\scriptscriptstyle R}}
\newcommand{\xiR}{\xi_{\scriptscriptstyle R}}
\newcommand{\ksym}{\mathsf k}
\newcommand{\pair}[1]{\langle\!\langle #1 \rangle\!\rangle}
\newtheorem{theorem}{Theorem}[section]
\newtheorem{proposition}[theorem]{Proposition}
\newtheorem{lemma}[theorem]{Lemma}

\newtheorem{remark}[theorem]{Remark}
\theoremstyle{definition}
\newtheorem{definition}[theorem]{Definition}
\newtheorem{assumption}[theorem]{Assumption}
\newtheorem{example}[theorem]{Example}
\theoremstyle{plain}

\title[Variational Nonlinearities and Wave Selection]{Variational Nonlinearities and Wave Selection at an $O(2)$-Hopf Bifurcation}

\author[Şengül]{Taylan Şengül}
\address{Department of Mathematics, Marmara University, 34722 Istanbul, Turkey}
\email{taylan.sengul@marmara.edu.tr}
\date{\today}

\begin{document}

\begin{abstract}
At an $O(2)$-equivariant Hopf bifurcation, the real parts $\zetaR$ and $\xiR$ of the two cubic normal-form coefficients govern selection between traveling and standing waves.
In a fourth-order regularized $p$-system, $\xiR$ was observed to vanish and the cubic nonlinearity to shift only frequencies.
We show that both observations are structural: variational nonlinearities enforce them even when the linear part is dissipative and non-Hamiltonian.
We consider a class of two-component PDEs on the circle whose quadratic and cubic nonlinearities are generated by reflection-invariant local Hamiltonians of the fields and finitely many of their spatial derivatives, through a common constant-coefficient Poisson operator.
Then $\xiR=0$ for every admissible linear part, while $\zetaR$ depends only on the quadratic nonlinearity.
For variational nonlinearities the direct cubic terms are purely imaginary.
The mixed cascade term loses its real part because its second harmonic has zero temporal frequency and the resolvent keeps its parity, while the self-cascade at the doubled critical frequency keeps a real part.
When $\zetaR\ne0$, the bifurcating traveling waves are saddles, while the standing waves are supercritical and orbitally asymptotically stable for $\zetaR<0$ and subcritical and unstable for $\zetaR>0$.
When the nonlinearities are the Poisson operator applied to polynomials in the fields alone, and its symbol does not vanish at the critical and second-harmonic wavenumbers, cancellation for every admissible linear part conversely characterizes variationality.
In a dissipatively regularized Boussinesq family, as dissipation tends to zero, $\zetaR$ vanishes linearly off resonance and diverges inversely at the $2{:}1$ resonance.
The Hamiltonian limit of the cubic coefficient is therefore singular.
\end{abstract}

\maketitle

\section{Introduction}

At an $O(2)$-equivariant Hopf bifurcation on the circle, the competing bifurcating states are waves that travel around the circle and waves that stand on it.
The choice is decided by the real parts $\zetaR$ and $\xiR$, generically independent~\cite{vangils1986hopf,crawford1991symmetry,crawford1988degenerate}, of the cubic coefficients of the $O(2)$-Hopf normal form~\eqref{reduced_system}.
Yao~\cite{yao20142} and Li and Yao~\cite{liyao2015} computed these coefficients for a fourth-order regularized $p$-system modeling cellular shock instability.
Their regularization supplies a dissipative, non-Hamiltonian linear part, while the $p$-system $u_t=v_x$, $v_t=\sigma(u)_x$ itself remains a Hamiltonian vector field.
They found $\xiR=0$ with $\zetaR\ne0$.
Li and Yao also noted, as surprising, that the cubic nonlinearity arising from the pressure law contributes only to the imaginary parts, shifting frequencies without affecting selection.
The closed-form coefficients derived in~\cite{ozersengultiryakioglu2026} for a broad class of two-component conservative systems showed the same pattern recurring well beyond the $p$-system, and the present paper grew out of that observation.

Theorem~\ref{thm:hamiltonian_flux} explains both.
It concerns the $O(2)$-equivariant system
\begin{equation}\label{main}
\partial_tU=L_\lambda U+G(U),
\qquad U=(u,v)^T,
\end{equation}
on the torus $\mathbb T=\mathbb R/2\pi\mathbb Z$ for mean-zero $U$, with $G_2$ and $G_3$ the quadratic and cubic parts of $G$.
Suppose the nonlinearities are $G_j=S\mathcal K\,\delta\mathsf H_{j+1}/\delta U$ for $j=2,3$, where $S=\begin{psmallmatrix}0&1\\1&0\end{psmallmatrix}$ exchanges the two components of $U$, which have opposite parity under the reflection, $\mathcal K=\partial_xP(-\partial_x^2)$ for a nonzero real polynomial $P$, and $\mathsf H_{j+1}$ is an arbitrary real reflection-invariant local Hamiltonian of degree $j+1$ in the sense of~\eqref{eq:local_hamiltonian}.
Suppose also that the linear part $L_\lambda$ is admissible in the sense of Section~\ref{sec:setting}.
Then $\xiR=0$, while $\zetaR$ is generically nonzero and comes from the quadratic nonlinearity alone.
When $\zetaR\ne0$, the traveling waves are saddles, while the standing waves are supercritical and orbitally asymptotically stable for $\zetaR<0$ and subcritical and unstable for $\zetaR>0$.
The models of~\cite{yao20142,liyao2015} become special cases of this theory, with $\mathcal K=\partial_x$.

The proof is a vertex calculus built on the extraction step of~\cite{ozersengultiryakioglu2026}.
Each cubic coefficient of the normal form is the sum of a direct term, obtained by evaluating $G_3$ on critical modes, and a cascade term, in which $G_2$ acts on a critical mode and a second harmonic that it generated.
Variationality makes both direct terms purely imaginary.
The zero temporal frequency of the second harmonic behind $\xi$ preserves the required resolvent parity and removes the cascade real part of $\xi$, whereas the $2\omega_c$ shift behind $\zeta$ breaks that parity and allows the cascade real part of $\zeta$ to survive.

Without variationality, equivariance and conservation alone do not produce the cancellation, and $\xiR=0$ is a single real condition on the nonlinearity.
Outside the variational class the two real parts are separate and model-dependent, as in reaction--diffusion~\cite{yaoliuwang2019} and delayed optical~\cite{budzinskiyrazgulin2017} systems, and the selection can even reverse: in a nonreciprocal two-field Swift--Hohenberg model the traveling waves are stable and the standing waves unstable at onset~\cite{tateyama2026higher}.
Vanishing of $\xiR$ at a single $L_\lambda$ can be accidental.
For derivative-free fluxes, $G_j=\mathcal KF_j$ with $F_j$ a polynomial in $U$ alone, Theorems~\ref{thm:hamiltonian_flux} and~\ref{thm:converse} give, when $P(m_c^2)P(4m_c^2)\ne0$ for the critical wavenumber $m_c$, and so always for $\mathcal K=\partial_x$,
\[
F_2,\,F_3\ \text{variational}
\quad\Longleftrightarrow\quad
\xiR=0\ \text{for every admissible }L_\lambda .
\]
Without this nondegeneracy the converse fails (Proposition~\ref{prop:degeneracies}).
For derivative-dependent fluxes only the forward implication survives: Example~\ref{ex:null_form} gives a non-variational flux that does not alter the cubic normal-form coefficients, so $\xiR=0$ holds for every admissible $L_\lambda$ although the flux is not variational.
See Problem~\ref{prob:kernel} of Section~\ref{sec:open}.

The perturbed-Hamiltonian families of Section~\ref{sec:boussinesq} are a PDE instance of Hamiltonian systems broken only in their linear part.
Two other subclasses of that class that are not covered in this paper are the damped gyroscopic ODE systems of dissipation-induced instability~\cite{bkmr1994,krechetnikovmarsden2007}, where damping destabilizes through two frequencies of opposite Krein signature, and the weakly damped parametrically forced waves of~\cite{martelknoblochvega2000}, where the forcing couples the counterpropagating amplitudes and so favors standing waves~\cite{rieckecrawfordknobloch1988}.
In both, the mechanism is linear, whereas here it is cubic.
Knobloch, Mahalov and Marsden identified the nondissipative limit of the $O(2)$-Hopf amplitude equations with a Hamiltonian normal form in which both wave amplitudes are conserved and only their frequencies shift~\cite{knoblochmahalovmarsden1994}.
They cautioned that in the Hamiltonian limit of the Faraday problem the number of participating eigenvalues doubles.
In the weakly damped Faraday problem every nonresonant real cubic coefficient is of the order of the damping~\cite{portersilber2002}.
For the perturbed-Hamiltonian families the cross-coupling coefficient has no real part at any damping, and in the formal nonresonant Hamiltonian limit the self-coupling coefficient loses its real part as well (Remark~\ref{rem:hamiltonian_limit}), so dissipation is the sole source of $\zetaR$ and of the saturation strength $\abs{\zetaR}$ it sets.
Along the perturbed-Hamiltonian Boussinesq family of Section~\ref{sec:boussinesq}, $\zetaR<0$ throughout, and the saturation strength is of the order of the damping off the $2{:}1$ resonance and of its inverse on it.

\section{Setting, hypotheses, and reduction}\label{sec:setting}

We work with~\eqref{main} in the mean-zero Sobolev scale $\mathcal H^s=\dot H^s_{\mathrm{per}}(0,2\pi)^2$.
We write $U_1=u$ and $U_2=v$ for the components of $U$.
The complexified $L^2$ pairing is $\langle U,V\rangle_{L^2}=\int_0^{2\pi}\langle U(x),V(x)\rangle_{\mathbb C^2}\,dx$, where $\langle z,w\rangle_{\mathbb C^2}=z\mathbin{\cdot}\overline w$ is the Hermitian pairing on $\mathbb C^2$.
The group $O(2)$ acts by translations $T_h$ and by the reflection
\begin{equation}\label{eq:reflection_action}
R\begin{bmatrix}u(x)\\v(x)\end{bmatrix}
=\begin{bmatrix}u(-x)\\-v(-x)\end{bmatrix}.
\end{equation}
We have $(RU)(x)=R_0U(-x)$ where $R_0=\operatorname{diag}(\varepsilon_1,\varepsilon_2)$ and
\begin{equation}\label{eq:reflection_signs}
\varepsilon_1=1,
\quad
\varepsilon_2=-1.
\end{equation}
The swap matrix $S=\begin{psmallmatrix}0&1\\1&0\end{psmallmatrix}$ anticommutes with $R_0$ and is the Poisson matrix of the variational nonlinearities introduced below.

Fix a critical parameter $\lambda_c\in\mathbb R$ and a critical wavenumber $m_c\in\mathbb Z_{\ge1}$.

\begin{assumption}[Linear part and parabolic structure]\phantomsection\label{ass:D}
For $\lambda$ near $\lambda_c$,
\[
(L_\lambda)_{ab}=\sum_{k=0}^{2m_L}c^{ab}_k(\lambda)\,\partial_x^k,
\qquad
a,b\in\{1,2\},
\]
with real coefficients $c^{ab}_k$ depending smoothly on $\lambda$, and $m_L\ge1$ least such, together with the following:
\begin{enumerate}[label=(D\arabic*), ref=(D\arabic*)]
\item\label{E} $L_\lambda$ is $O(2)$-equivariant.
\item\label{D} Diagonal and nondegenerate principal part: $c^{11}_{2m_L}c^{22}_{2m_L}\ne0$ and $c^{12}_{2m_L}=c^{21}_{2m_L}=0$.
\item\label{P} $r_\lambda\le2m_L-1$, where $r_\lambda=\max\{k:\partial_\lambda c^{ab}_k\not\equiv0\text{ near }\lambda_c\text{ for some }a,b\}$.
\end{enumerate}
\end{assumption}

Items~\ref{D} and~\ref{P} enter only through the center-manifold reduction of Proposition~\ref{thm:reduction}, where they are the hypotheses of~\cite{ozersengultiryakioglu2026}.

We write $M_m$ for the Fourier symbol of $L_\lambda$ at mode $m$, obtained by $\partial_x\mapsto im$.
By Assumption~\ref{ass:D}\ref{E} and the reality of the coefficients it satisfies the parity
\begin{equation}\label{eq:symbol_parity}
R_0M_mR_0=M_{-m}=\overline{M_m},
\end{equation}
which is in turn equivalent to~\ref{E}.

\begin{assumption}[Structure and differential order]\phantomsection\label{ass:S}
The nonlinearity $G=(g_1,g_2)^T$ has real coefficients and satisfies $G(0)=DG(0)=0$, together with the following:
\begin{enumerate}[label=(S\arabic*), ref=(S\arabic*)]
\item\label{S1} $G$ is a differential polynomial.
\item\label{S2} $G$ is $O(2)$-equivariant.
\item\label{S3} $G$ is conservative, $G=\partial_xf$.
\item\label{S4} $r_G\le2m_L-1$ holds, where $r_G$ is the highest derivative order appearing in $G$, with $r_G=0$ if $G\equiv0$.
\end{enumerate}
\end{assumption}

We write $G=G_2+G_3+\cdots$ where
\begin{equation}\label{eq:multilinear_maps}
G_2(U)=\tfrac12D^2G(0)(U,U),
\qquad
G_3(U)=\tfrac16D^3G(0)(U,U,U).
\end{equation}

\begin{assumption}[Simple $O(2)$-Hopf crossing]\phantomsection\label{ass:H}\leavevmode
\begin{enumerate}[label=(H\arabic*), ref=(H\arabic*)]
\item\label{H1} $\tr M_{m_c}(\lambda_c)=0$ and $\det M_m(\lambda_c)>0$ for all $m\ne0$.
\item\label{H2} Transversality: $\frac{d}{d\lambda}\tr M_{m_c}(\lambda_c)\ne0$.
\item\label{H3} $\tr M_m(\lambda_c)<0$ for $m\in\mathbb Z\setminus\{0,\pm m_c\}$.
\item\label{H4} Spectral gap: $\liminf_{|m|\to\infty}\det M_m(\lambda_c)/\abs{\tr M_m(\lambda_c)}>0$.
\end{enumerate}
\end{assumption}

\begin{definition}
A pair $(L_\lambda,G)$ is \emph{admissible}, and $L_\lambda$ is \emph{admissible for $G$}, if Assumptions~\ref{ass:D},~\ref{ass:S} and~\ref{ass:H} hold.
We say simply that $L_\lambda$ is \emph{admissible} when $G$ is clear from context.
\end{definition}

\begin{proposition}[Existence of admissible linear parts]\label{prop:witness}
For every $G$ satisfying Assumption~\ref{ass:S}\ref{S1}--\ref{S3} and every $m_c\ge1$ there is an $L_\lambda$ admissible for $G$ with critical wavenumber $m_c$.
One is $L_\lambda=\tau_\lambda I+S\partial_x$ with $\tau_\lambda=\lambda-\nu(-\partial_x^2-m_c^2)^N$, $\nu>0$ and $N$ even with $2N-1\ge r_G$.
\end{proposition}
The proof is Lemma~\ref{lem:witness} of Section~\ref{sec:defects}, which exhibits a three-parameter family of such operators.
The one displayed is its member $a=0$, $r=1$.

Let $\beta_{m,n}$, $n=1,2$, be the eigenvalues of $M_m$.
For the critical modes $m=\pm m_c$ they are simple by Assumption~\ref{ass:H}\ref{H1}, and we write $\mathbf q_{m,n}$ for the corresponding eigenvectors.
Noncritical blocks may be defective and enter only through resolvents.
By Assumption~\ref{ass:H}\ref{H1} the critical eigenvalues can be labeled as $\beta_{m_c,1}(\lambda_c)=i\omega_c$ and $\beta_{m_c,2}(\lambda_c)=-i\omega_c$, where
\begin{equation}\label{eq:critical_frequency}
\omega_c=\sqrt{\det M_{m_c}(\lambda_c)}>0.
\end{equation}
The critical adjoint eigenfunctions are, for $m=\pm m_c$,
\begin{equation}\label{eq:adjoint_modes}
e_{m,n}^*(x)=e^{imx}\mathbf q_{m,n}^*,
\qquad
M_{-m}^T\mathbf q_{m,n}^*=\beta_{-m,n}\mathbf q_{m,n}^*,
\qquad
\langle\mathbf q_{m,n},\mathbf q_{m,n'}^*\rangle_{\mathbb C^2}=\frac{\delta_{nn'}}{2\pi}.
\end{equation}
With this normalization, let
\begin{equation}\label{eq:critical_projection}
\mathcal P_nF
:=\langle F,e_{m_c,n}^*\rangle_{L^2}.
\end{equation}

We write the first critical eigenvector as
\begin{equation}\label{eq:qc}
\mathbf q:=\mathbf q_{m_c,1}=(u_c,v_c)^T,
\qquad
r_c=\Re(u_c\overline{v_c}),
\qquad
\chi=\Im(u_c\overline{v_c}).
\end{equation}

We may choose the second critical eigenvector as
\begin{equation}\label{eq:critical_phase_convention}
\mathbf q_B:=\mathbf q_{m_c,2}
=-R_0\overline{\mathbf q}
=(-\overline{u_c},\overline{v_c})^T.
\end{equation}
We also choose $\mathbf q_{-m_c,n}=\overline{\mathbf q_{m_c,n}}$.
By independence, $\det[\mathbf q_{m_c,1},\mathbf q_{m_c,2}]=2r_c\ne0$, and the biorthogonality in~\eqref{eq:adjoint_modes} gives
\begin{equation}\label{eq:ham_adjoint}
\mathbf q_{m_c,1}^*=\frac{S\mathbf q_{m_c,1}}{4\pi r_c}.
\end{equation}

Under Assumption~\ref{ass:H} the center subspace at $\lambda_c$ is four-dimensional, parametrized by the critical amplitudes $A=\mathcal P_1U$ and $B=\mathcal P_2U$,
\begin{equation}\label{eq:center_linear_parametrization}
\psi_c(A,B)
=Ae^{im_cx}\mathbf q+Be^{im_cx}\mathbf q_B
+\overline A e^{-im_cx}\overline{\mathbf q}
+\overline B e^{-im_cx}\overline{\mathbf q_B}.
\end{equation}

\begin{proposition}[Center-manifold reduction]
\label{thm:reduction}
Fix an integer $k\ge5$.
For an admissible pair $(L_\lambda,G)$ there exist, in a neighborhood of the origin in $\mathcal H^{2m_L}$ and for $\lambda$ in a neighborhood of $\lambda_c$ depending on $k$:
\begin{enumerate}[label=(\alph*)]
\item a locally invariant, $O(2)$-equivariant four-dimensional center manifold of class $C^k$, which attracts exponentially every solution that remains in that neighborhood;
\item\label{item:normal_form} a near-identity transformation, with reduced dynamics governed by the normal-form equations
\begin{equation}\label{reduced_system}
\begin{aligned}
\dot Z_1&=\beta_{m_c,1}(\lambda)Z_1+Z_1\bigl(\zeta\abs{Z_1}^2+\xi\abs{Z_2}^2\bigr)+O(\abs{\lambda-\lambda_c}\abs Z^3+\abs Z^5),\\
\dot Z_2&=\beta_{m_c,2}(\lambda)Z_2+Z_2\bigl(\bar\xi\abs{Z_1}^2+\bar\zeta\abs{Z_2}^2\bigr)+O(\abs{\lambda-\lambda_c}\abs Z^3+\abs Z^5),
\end{aligned}
\end{equation}
\end{enumerate}
Here $\abs Z=\bigl(\abs{Z_1}^2+\abs{Z_2}^2\bigr)^{1/2}$.
The cubic coefficients split as
\begin{equation}\label{eq:coefficient_splitting}
\zeta=\hat s_{1,1}+\hat c_{1,1},
\qquad
\xi=\hat s_{1,2}+\hat c_{1,2},
\end{equation}
with $\hat s_{1,j}$ and $\hat c_{1,j}$ given by~\eqref{eq:direct_coefficient_extraction}--\eqref{eq:cascade_coefficient_extraction}, all evaluated at $\lambda=\lambda_c$.
\end{proposition}

\begin{proof}
The reduction argument of~\cite[Thms.~3.1 and~3.2]{ozersengultiryakioglu2026} applies under admissibility, in the triple $\mathcal Z=\mathcal H^{2m_L}\hookrightarrow\mathcal Y=\mathcal H^{2m_L-r}\hookrightarrow\mathcal X=\mathcal H^0$ with $r=\max(r_G,r_\lambda)\le2m_L-1$, and gives parts~(a) and~\ref{item:normal_form} except for the local attraction.
The splitting~\eqref{eq:coefficient_splitting} is derived in~\S\ref{subsec:extraction}.
Attraction follows from that verification with~\ref{H1} and~\ref{H3}, which put no spectrum of $L_{\lambda_c}$ in the open right half-plane.
The resolvent estimate of~\cite[\S6.1]{ozersengultiryakioglu2026} for Hypothesis~2.15 of~\cite{haragus2010local} then gives Hypothesis~3.20 there by~\cite[Rem.~B.2]{haragus2010local}, and~\cite[Thm.~3.23]{haragus2010local} makes the center manifold locally attracting.
Since the manifold is parameter-dependent, the attraction is uniform for $\lambda$ near $\lambda_c$ and does not need the critical pair to stay in the closed left half-plane.
\end{proof}

We write $\zetaR=\Re\zeta$ and $\xiR=\Re\xi$.
These real parts govern the radial dynamics of the cubic normal form.

We write $h_2$ for the quadratic part of the center-manifold expansion and $[\mathfrak m]$ for the extraction of the total coefficient of the monomial $\mathfrak m$~\cite{knuth1994bracket}.
The \emph{direct} coefficients $\hat s_{1,1},\hat s_{1,2}$, obtained by evaluating $G_3$ on critical modes, and the \emph{cascade} coefficients $\hat c_{1,1},\hat c_{1,2}$, in which $G_2$ acts on a critical mode and a second harmonic that it generated, are then
\begin{align}
\hat s_{1,1}
&=\mathcal P_1\!\left([A^2\overline A]G_3(\psi_c)\right),
&
\hat s_{1,2}
&=\mathcal P_1\!\left([AB\overline B]G_3(\psi_c)\right),
\label{eq:direct_coefficient_extraction}\\
\hat c_{1,1}
&=\mathcal P_1\!\left([A^2\overline A]\bigl(DG_2(\psi_c)h_2\bigr)\right),
&
\hat c_{1,2}
&=\mathcal P_1\!\left([AB\overline B]\bigl(DG_2(\psi_c)h_2\bigr)\right),
\label{eq:cascade_coefficient_extraction}
\end{align}
all four being derived in~\S\ref{subsec:extraction}.
Reference~\cite{ozersengultiryakioglu2026} calls the same four quantities the self- and cross-interaction coefficients, self referring to the critical mode acting on itself and cross to its interaction with the higher modes, so that self-interaction is our direct term and cross-interaction our cascade term.
Here self- and cross-coupling are reserved for the $A$-versus-$B$ distinction and name $\zeta$ and $\xi$ themselves, as in~\cite{martelknoblochvega2000}.

We now fix the vocabulary of the main results.
A \emph{multiplier} is an operator
\[
\mathcal K=\partial_xP(-\partial_x^2)
\]
with $P$ a nonzero real polynomial.
Given a multiplier $\mathcal K$, a \emph{flux} of $G_j$, $j\ge2$, is a homogeneous differential polynomial map $F_j$ of degree $j$ with $G_j=\mathcal KF_j$.
For $G$ satisfying~\ref{S1}--\ref{S3}, each $G_j$ has a flux with $\mathcal K=\partial_x$, that is, $P\equiv1$.
A \emph{homogeneous local Hamiltonian of degree $n$} is a functional
\begin{equation}\label{eq:local_hamiltonian}
\mathsf H_n(U)
=\int_{\mathbb T}
H_n\bigl(U,\partial_xU,\ldots,\partial_x^rU\bigr)\,dx,
\end{equation}
where $r\ge0$ is an integer and $H_n$ is a real homogeneous polynomial of total degree $n$.
Such a functional is translation invariant and $R$-invariant when $\mathsf H_n(RU)=\mathsf H_n(U)$.
Its \emph{variational derivative} is the vector $\delta\mathsf H_n/\delta U=(\delta\mathsf H_n/\delta U_1,\delta\mathsf H_n/\delta U_2)^T$ defined by
\begin{equation}\label{eq:variational_derivative_definition}
D\mathsf H_n(U)[W]
=\int_{\mathbb T}
\frac{\delta\mathsf H_n}{\delta U}(U)\mathbin{\cdot}W\,dx
\qquad\text{for every smooth $2\pi$-periodic }W=(W_1,W_2)^T.
\end{equation}
We call $F_j$, and hence $G_j$, \emph{variational} if
\begin{equation}\label{eq:ham_variational_field}
F_j=S\frac{\delta\mathsf H_{j+1}}{\delta U},
\qquad
G_j=S\mathcal K\frac{\delta\mathsf H_{j+1}}{\delta U},
\end{equation}
for an $R$-invariant homogeneous local Hamiltonian $\mathsf H_{j+1}$ of degree $j+1$.
Since every nonzero constant symmetric matrix $A$ anticommuting with $R_0$ is a multiple of $S$, \eqref{eq:ham_variational_field} is, up to a scalar absorbed into $\mathcal K$, the general form of an $R$-equivariant Hamiltonian vector field with Poisson operator $A\mathcal K$ and an $R$-invariant Hamiltonian.
We call $F_j$, and hence $G_j$, \emph{derivative-free} if $F_j$ is a homogeneous polynomial map of degree $j$ in $U$ alone.

Two degeneracies of the multiplier annihilate some or all of the cubic coefficients, independently of any variational hypothesis.
\begin{proposition}[Multiplier degeneracies]\label{prop:degeneracies}
Let $(L_\lambda,G)$ be admissible and let $G_2$, $G_3$ have fluxes $F_2$, $F_3$ with the multiplier $\mathcal K=\partial_xP(-\partial_x^2)$.
\begin{enumerate}[label=(\roman*)]
\item\label{Dg1} If $P(m_c^2)=0$, then $\hat s_{1,1}=\hat s_{1,2}=\hat c_{1,1}=\hat c_{1,2}=0$, so $\zeta=\xi=0$.
\item\label{Dg2} If $P(4m_c^2)=0$, then $\hat c_{1,1}=\hat c_{1,2}=0$, so $\zeta=\hat s_{1,1}$ and $\xi=\hat s_{1,2}$.
\end{enumerate}
\end{proposition}
The proposition is proved in~\S\ref{sec:vertex}.

\begin{remark}\label{rem:poisson}
The constant-coefficient skew-adjoint operator $S\mathcal K$ is a possibly degenerate Poisson operator~\cite[Ch.~7]{olver1993applications}, and for $\mathcal K=\partial_x$ it is a flat bracket of Dubrovin and Novikov~\cite{dubrovinnovikov1989}.
\end{remark}

\section{Main results}\label{sec:main}

\begin{theorem}[Cancellation for variational nonlinearities]
\label{thm:hamiltonian_flux}
Suppose $(L_\lambda,G)$ is admissible and $G_2$, $G_3$ are variational for a common multiplier $\mathcal K$.
Then:
\begin{enumerate}[label=(\roman*)]
\item\label{T1} Both direct coefficients are purely imaginary,
\[
\Re\hat s_{1,1}=\Re\hat s_{1,2}=0.
\]
\item\label{T2} The mixed cascade coefficient is purely imaginary,
\[
\Re\hat c_{1,2}=0.
\]
\item\label{T3} Consequently
\[
\xiR=0,
\qquad
\zetaR=\Re\hat c_{1,1},
\]
and if $P(m_c^2)P(4m_c^2)=0$ then also $\zetaR=0$.
\item\label{T4} Let $\mu=\Re\beta_{m_c,1}(\lambda)$, so that the trivial solution is stable for $\mu<0$ and unstable for $\mu>0$, and suppose $\zetaR\ne0$.
For small nonzero $\mu$, both the traveling-wave and standing-wave branches exist exactly for $\mu\zetaR<0$.
In that case the traveling waves are saddles, while the standing wave is supercritical and orbitally asymptotically stable in $\mathcal H^{2m_L}$ for $\zetaR<0$, subcritical and unstable for $\zetaR>0$.
\end{enumerate}
\end{theorem}

Theorem~\ref{thm:hamiltonian_flux} is proved in~\S\ref{sec:vertex}.
Theorem~\ref{thm:converse} is a converse to it within the derivative-free class introduced in Section~\ref{sec:setting}.

\begin{theorem}[Uniform cancellation characterizes variationality]\label{thm:converse}
Suppose that~\ref{S1}--\ref{S3} of Assumption~\ref{ass:S} hold, that $G_2$, $G_3$ are derivative-free for a multiplier $\mathcal K=\partial_xP(-\partial_x^2)$, and fix an integer $m_c\ge1$ with
\begin{equation}\label{eq:multiplier_nondegeneracy}
P(m_c^2)P(4m_c^2)\ne0.
\end{equation}
Then the following are equivalent:
\begin{enumerate}[label=(\roman*)]
\item $G_2$ and $G_3$ are variational for $\mathcal K$;
\item $\xiR=0$ for every $L_\lambda$ admissible for $G$ with critical wavenumber $m_c$.
\end{enumerate}
Without~\eqref{eq:multiplier_nondegeneracy}, (ii) does not imply~(i).
\end{theorem}
The proof of Theorem~\ref{thm:converse} is given in Section~\ref{sec:defects}.
The quantifier in~(ii) is never vacuous: Proposition~\ref{prop:witness} supplies admissible linear parts with the prescribed critical wavenumber for every $G$ satisfying~\ref{S1}--\ref{S3}.
Vanishing of $\xiR$ for a single $L_\lambda$ does not imply variationality.
The bilinear family of~\cite[\S5.3, Eq.~(47)]{ozersengultiryakioglu2026} is an instance.
Example~\ref{ex:null_form} shows that the derivative-free restriction cannot be dropped.

\section{Vertex calculus and proof of the cancellation theorem}\label{sec:methods}

The cubic coefficients are computed at $\lambda=\lambda_c$, so from here on we write $M_m$ for $M_m(\lambda_c)$, displaying the parameter where it varies.

\subsection{Fourier conventions}

For a scalar or vector-valued periodic function we use
\begin{equation}\label{eq:fourier_convention}
\widehat U(k)=\frac1{2\pi}\int_0^{2\pi}U(x)e^{-ikx}\,dx,
\qquad
U(x)=\sum_{k\in\mathbb Z}\widehat U(k)e^{ikx}.
\end{equation}
For a multiplier $\mathcal K=\partial_xP(-\partial_x^2)$, write
\begin{equation}\label{eq:K_symbol}
\widehat{\mathcal Kf}(m)=i\ksym(m)\widehat f(m),
\qquad
\ksym(m)=mP(m^2)\in\mathbb R,
\qquad
\ksym(-m)=-\ksym(m).
\end{equation}

\subsection{Critical modes and temporal eigenvalues}

We write
\[
\beta_1=\beta_{m_c,1}(\lambda_c)=\overline{\beta_{m_c,2}(\lambda_c)}=i\omega_c.
\]
To each amplitude variable $X\in\{A,B,\overline A,\overline B\}$ we associate the wavenumber $m_X$ and temporal eigenvalue $\beta_X$ it carries in~\eqref{eq:center_linear_parametrization}: $m_A=m_B=m_c$, $m_{\overline A}=m_{\overline B}=-m_c$, $\beta_A=\beta_{\overline B}=\beta_1$, and $\beta_B=\beta_{\overline A}=\overline{\beta_1}$.
To a quadratic monomial $XY$ we associate $m_{XY}=m_X+m_Y$ and $\beta_{XY}=\beta_X+\beta_Y$:
\begin{equation}\label{eq:quadratic_mode_frequency_table}
\begin{array}{c|ccc|ccc}
XY&AA&AB&BB&A\overline A&A\overline B&B\overline B\\ \hline
m_{XY}&2m_c&2m_c&2m_c&0&0&0\\
\beta_{XY}&2\beta_1&0&2\overline{\beta_1}&0&2\beta_1&0
\end{array}
\end{equation}

\subsection{The quadratic center-manifold equation}

We parametrize the center manifold as
\[
U=\psi_c+h_2(A,B,\overline A,\overline B)+O(|(A,B)|^3),
\qquad \mathcal P_nh_2=0,\quad n=1,2.
\]
For a quadratic monomial $XY$, define its second-harmonic and forcing vectors by
\begin{equation}\label{eq:quadratic_forcing_definition}
\Phi_{XY}e^{im_{XY}x}
:=[XY]h_2,
\qquad
\hat g_{XY}e^{im_{XY}x}
:=[XY]G_2(\psi_c).
\end{equation}
The homological equations~\cite{haragus2010local} are
\begin{equation}\label{eq:quadratic_homological_equation}
\left(\beta_{XY}I-M_{m_{XY}}\right)\Phi_{XY}
=\hat g_{XY}.
\end{equation}
Only two of them enter the cubic extraction of~\S\ref{subsec:extraction}:
\begin{align}
\Phi_{AA}&=(2\beta_1I-M_{2m_c})^{-1}\hat g_{AA},
\label{eq:self_second_harmonic}\\
\Phi_{AB}&=(-M_{2m_c})^{-1}\hat g_{AB},
\label{eq:mixed_second_harmonic}
\end{align}
where both resolvents exist by Assumption~\ref{ass:H}.

Because the center manifold is constructed in the mean-zero phase space, $h_2$ has no zero Fourier mode, so
\begin{equation}\label{eq:zero_mode_vanishing}
\Phi_{A\overline A}=\Phi_{A\overline B}=\Phi_{B\overline B}=0.
\end{equation}

\subsection{Derivation of the cubic extraction formulas \texorpdfstring{\eqref{eq:direct_coefficient_extraction}--\eqref{eq:cascade_coefficient_extraction}}{}}\label{subsec:extraction}

We define the cubic field
\begin{equation}\label{eq:cubic_reduced_field}
\mathcal R_3(\psi_c)
=G_3(\psi_c)+DG_2(\psi_c)h_2
=\tfrac16D^3G(0)(\psi_c,\psi_c,\psi_c)+D^2G(0)(\psi_c,h_2).
\end{equation}
Extraction is linear and amplitude-independent, so it commutes with the critical projection $\mathcal P_1$ of~\eqref{eq:critical_projection} and with $S\mathcal K$,
\begin{equation}\label{eq:extraction_commutation}
\mathcal P_1([\mathfrak m]F)=[\mathfrak m]\mathcal P_1(F),
\qquad
[\mathfrak m](S\mathcal KF)=S\mathcal K([\mathfrak m]F),
\end{equation}
for every amplitude monomial $\mathfrak m$ and amplitude-polynomial field $F$.
The reduced system has no quadratic part, its cubic part is $[\mathcal P_1\mathcal R_3(\psi_c),\mathcal P_2\mathcal R_3(\psi_c)]^T$, and the near-identity transformation $(A,B)\mapsto(Z_1,Z_2)$ leaves the resonant coefficients unchanged.
The $\mathcal P_1$ resonances are $A^2\overline A$ and $AB\overline B$.
Reflection equivariance makes the $\mathcal P_2$ resonances their conjugates~\cite{vangils1986hopf,crawford1991symmetry}.
Extracting these two monomials from the two summands of~\eqref{eq:cubic_reduced_field} gives the four coefficients of~\eqref{eq:coefficient_splitting}, namely~\eqref{eq:direct_coefficient_extraction}--\eqref{eq:cascade_coefficient_extraction}~\cite{haragus2010local}.

\subsection{Vertices and reality}

For $\mathsf H_n$ as in~\eqref{eq:local_hamiltonian}, define its \emph{vertex} $T^{(n)}$ as the unique complex-multilinear extension of the symmetric real form $D^n\mathsf H_n(0)$,
\begin{equation}\label{eq:ham_vertex_multilinear}
\begin{aligned}
D^n\mathsf H_n(0)[U^{(1)},\ldots,U^{(n)}]
={}&\sum_{a_1,\ldots,a_n=1}^2
\sum_{m_1+\cdots+m_n=0}
T^{(n)}_{a_1\cdots a_n}(m_1,\ldots,m_n)
\prod_{\ell=1}^n\widehat{U^{(\ell)}}_{a_\ell}(m_\ell).
\end{aligned}
\end{equation}
We call $(a_\ell,m_\ell)$ the $\ell$th \emph{leg} of $T^{(n)}$.
Setting $U^{(1)}=\cdots=U^{(n)}=U$ and using the degree-$n$ homogeneity of $\mathsf H_n$ gives
\begin{equation}\label{eq:ham_vertex}
\mathsf H_n(U)
=\frac{1}{n!}
\sum_{a_1,\ldots,a_n=1}^2
\sum_{m_1+\cdots+m_n=0}
T^{(n)}_{a_1\cdots a_n}(m_1,\ldots,m_n)
\prod_{\ell=1}^n\widehat U_{a_\ell}(m_\ell).
\end{equation}
The \emph{Fourier-momentum conservation} constraint $m_1+\cdots+m_n=0$ follows from orthogonality.
\begin{lemma}\label{lem:hamiltonian_vertex_identities}
Let $\mathsf H_n$ be as in~\eqref{eq:local_hamiltonian}.
For the reflection identity~\eqref{eq:vertex_reflection}, assume in addition that $\mathsf H_n$ is $R$-invariant.
For $a_1,\ldots,a_n\in\{1,2\}$, $m_1,\ldots,m_n\in\mathbb Z$ with $m_1+\cdots+m_n=0$, and any permutation~$\pi\in S_n$, the vertex satisfies
\begin{align}
T^{(n)}_{a_{\pi(1)}\cdots a_{\pi(n)}}
(m_{\pi(1)},\ldots,m_{\pi(n)})
&=T^{(n)}_{a_1\cdots a_n}(m_1,\ldots,m_n),
\label{eq:vertex_symmetry}\\
\overline{T^{(n)}_{a_1\cdots a_n}(m_1,\ldots,m_n)}
&=T^{(n)}_{a_1\cdots a_n}(-m_1,\ldots,-m_n),
\label{eq:vertex_conjugation}\\
T^{(n)}_{a_1\cdots a_n}(-m_1,\ldots,-m_n)
&=\left(\prod_{\ell=1}^n\varepsilon_{a_\ell}\right)
T^{(n)}_{a_1\cdots a_n}(m_1,\ldots,m_n).
\label{eq:vertex_reflection}
\end{align}
\end{lemma}

\begin{proof}
Identity~\eqref{eq:vertex_symmetry} holds by construction because $D^n\mathsf H_n(0)$ is symmetric.
Since $\mathsf H_n(U)$ is real for real $U$, conjugating~\eqref{eq:ham_vertex}, using $\overline{\widehat U_a(k)}=\widehat U_a(-k)$, and comparing coefficients gives~\eqref{eq:vertex_conjugation}.
Finally, \eqref{eq:reflection_action} and~\eqref{eq:reflection_signs} give, for $a\in\{1,2\}$ and $k\in\mathbb Z$,
\[
\widehat{(RU)}_a(k)
=\widehat{\varepsilon_aU_a(-\,\cdot\,)}(k)
=\varepsilon_a\widehat U_a(-k);
\]
substituting into $\mathsf H_n(RU)=\mathsf H_n(U)$ and comparing coefficients gives~\eqref{eq:vertex_reflection}.
\end{proof}

\begin{example}\label{ex:vertex_cubic}
For $\mathsf H_3(U)=\int_{\mathbb T}u^2v_x\,dx$, which is $R$-invariant,
\[
\mathsf H_3(U)
=2\pi i\!\!\sum_{m_1+m_2+m_3=0}\!\!
m_3\,\widehat U_1(m_1)\widehat U_1(m_2)\widehat U_2(m_3),
\]
so~\eqref{eq:ham_vertex} gives $T^{(3)}_{112}(m_1,m_2,m_3)=4\pi im_3$, with $T^{(3)}_{121}$ and $T^{(3)}_{211}$ its permutations by~\eqref{eq:vertex_symmetry} and every other component zero.
The vertex is purely imaginary, and~\eqref{eq:vertex_conjugation} holds because $m_3$ is odd.
Its variational derivative $(2uv_x,-2uu_x)^T$ is $R$-equivariant, as Lemma~\ref{lem:variational_admissible} guarantees.
\end{example}

\begin{lemma}\label{lem:variational_admissible}
For $j=2,3$, let $\mathsf H_{j+1}$ be given by~\eqref{eq:local_hamiltonian} and be $R$-invariant.
Then $G_j$ given by~\eqref{eq:ham_variational_field} satisfies Assumption~\ref{ass:S}\ref{S1}--\ref{S3}.
\end{lemma}

\begin{proof}
Since $\delta\mathsf H_{j+1}/\delta U$ is a differential polynomial and $G_j=\partial_x\bigl(SP(-\partial_x^2)\,\delta\mathsf H_{j+1}/\delta U\bigr)$, \ref{S1} and~\ref{S3} hold.
For~\ref{S2}, $G_j$ commutes with translations because the density does not depend on $x$ explicitly, and differentiating $\mathsf H_{j+1}(RU)=\mathsf H_{j+1}(U)$ with $\langle RU,RV\rangle_{L^2}=\langle U,V\rangle_{L^2}$ gives
\[
\frac{\delta\mathsf H_{j+1}}{\delta U}(RU)
=R\frac{\delta\mathsf H_{j+1}}{\delta U}(U).
\]
Both $\mathcal K$ and $S$ anticommute with $R$, the former because it contains only odd derivatives and the latter because $SR_0=-R_0S$, so $S\mathcal KR=R\,S\mathcal K$ and $G_j(RU)=RG_j(U)$.
\end{proof}

One further consequence of the vertex representation is the Fourier form of the variational derivative.

\begin{lemma}\label{lem:monomial_fourier_coefficient}
Let $\mathsf H_n$ be a homogeneous local Hamiltonian of degree $n$ with vertex $T^{(n)}$, as represented in~\eqref{eq:ham_vertex}.
Then, for $a\in\{1,2\}$ and $k\in\mathbb Z$,
\begin{equation}\label{eq:vertex_variational_derivative}
\begin{aligned}
\mathcal F_x\!\left[\frac{\delta\mathsf H_n}{\delta U_a}(U)\right](k)
={}&\frac{1}{2\pi(n-1)!}
\sum_{a_1,\ldots,a_{n-1}=1}^2
\sum_{m_1+\cdots+m_{n-1}=k}\\
&\quad{}\times
T^{(n)}_{a a_1\cdots a_{n-1}}(-k,m_1,\ldots,m_{n-1})
\prod_{\ell=1}^{n-1}\widehat U_{a_\ell}(m_\ell),
\end{aligned}
\end{equation}
with $\mathcal F_x$ the Fourier transform in~\eqref{eq:fourier_convention}.
\end{lemma}

\begin{proof}
Varying $U\mapsto U+\varepsilon W$ in~\eqref{eq:ham_vertex}, using total symmetry~\eqref{eq:vertex_symmetry} to move the leg carrying $W$ to the first slot, so that the $n$ terms of the product rule coincide and $n/n!=1/(n-1)!$, and comparing coefficients of $\widehat W_a(-k)$ with $D\mathsf H_n(U)[W]=2\pi\sum_{a,k}\mathcal F_x[\delta\mathsf H_n/\delta U_a](k)\,\widehat W_a(-k)$, which follows from~\eqref{eq:variational_derivative_definition} and Fourier orthogonality, gives~\eqref{eq:vertex_variational_derivative}.
\end{proof}

\subsection{Proof of Theorem~\ref{thm:hamiltonian_flux}}\label{sec:vertex}

Throughout this subsection, except in the proof of Proposition~\ref{prop:degeneracies}, $(L_\lambda,G)$ is admissible and $G_j=S\mathcal K\,\delta\mathsf H_{j+1}/\delta U$ for $j=2,3$, as in~\eqref{eq:ham_variational_field}.

\emph{Projection identity.}
For any smooth periodic $W:\mathbb T\to\mathbb C^2$,
\begin{equation}\label{eq:swap_cancellation}
\mathcal P_1(\mathcal K W)
=\left\langle \mathcal K W,e_{m_c,1}^*\right\rangle_{L^2}
=\frac{i\ksym(m_c)}{2r_c}\,\widehat W(m_c)\mathbin{\cdot}S\overline{\mathbf q},
\end{equation}
where the pairing selects the mode $m_c$ and we use~\eqref{eq:ham_adjoint}.

Defining
\[
Z_{\mathrm s,11}=[A^2\overline A]\frac{\delta\mathsf H_4}{\delta U}(\psi_c),
\qquad
Z_{\mathrm s,12}=[AB\overline B]\frac{\delta\mathsf H_4}{\delta U}(\psi_c),
\qquad
Z_{\mathrm c,12}=[AB\overline B]\,
D\!\left(\frac{\delta\mathsf H_3}{\delta U}\right)(\psi_c)h_2,
\]
and applying~\eqref{eq:swap_cancellation} to the equations~\eqref{eq:direct_coefficient_extraction}--\eqref{eq:cascade_coefficient_extraction} gives
\begin{equation}\label{eq:vertex_coefficient_form}
\hat s_{1,1}=\frac{i\ksym(m_c)}{2r_c}\,\widehat Z_{\mathrm s,11}(m_c)\mathbin{\cdot}\overline{\mathbf q},
\qquad
\hat s_{1,2}=\frac{i\ksym(m_c)}{2r_c}\,\widehat Z_{\mathrm s,12}(m_c)\mathbin{\cdot}\overline{\mathbf q},
\qquad
\hat c_{1,2}=\frac{i\ksym(m_c)}{2r_c}\,\widehat Z_{\mathrm c,12}(m_c)\mathbin{\cdot}\overline{\mathbf q},
\end{equation}
where we use~\eqref{eq:ham_variational_field}, \eqref{eq:extraction_commutation} and $S^TS=I$.

By $m_{AB}=2m_c$ in~\eqref{eq:quadratic_mode_frequency_table} and by~\eqref{eq:zero_mode_vanishing},
\[
[AB]h_2=\Phi_{AB}e^{2im_cx},
\qquad
[A\overline B]h_2=[B\overline B]h_2=0 .
\]
Since $\delta\mathsf H_3/\delta U$ is quadratic,
\[
Z_{\mathrm c,12}
=D\!\left(\frac{\delta\mathsf H_3}{\delta U}\right)\!
\bigl([\overline B]\psi_c\bigr)\,\Phi_{AB}e^{2im_cx},
\qquad
[\overline B]\psi_c=-R_0\mathbf q\,e^{-im_cx},
\]
the input legs $-m_c$ and $2m_c$ summing to the output mode $m_c$.
Substituting~\eqref{eq:center_linear_parametrization} into~\eqref{eq:vertex_variational_derivative} for the two direct fields, and differentiating~\eqref{eq:vertex_variational_derivative} for the cascade field, then contracting with $\overline{\mathbf q}$, gives
\begin{equation}\label{eq:vertex_cascade_contraction}
\begin{gathered}
\widehat Z_{\mathrm s,11}(m_c)\mathbin{\cdot}\overline{\mathbf q}
=\kappa_{\mathrm s,11}X_{\mathrm s,11},
\qquad
\widehat Z_{\mathrm s,12}(m_c)\mathbin{\cdot}\overline{\mathbf q}
=\kappa_{\mathrm s,12}X_{\mathrm s,12},
\qquad
\widehat Z_{\mathrm c,12}(m_c)\mathbin{\cdot}\overline{\mathbf q}
=\kappa_{\mathrm c,12}X_{\mathrm c,12},\\
\begin{aligned}
X_{\mathrm s,11}
&=\sum_{a,b,c,d}T^{(4)}_{abcd}(-m_c,m_c,m_c,-m_c)
\overline q_aq_bq_c\overline q_d,\\
X_{\mathrm s,12}
&=\sum_{a,b,c,d}T^{(4)}_{abcd}(-m_c,m_c,m_c,-m_c)
\overline q_aq_b\,\varepsilon_c\overline q_c\,\varepsilon_dq_d,\\
X_{\mathrm c,12}
&=-\sum_{a,b,c}T^{(3)}_{abc}(-m_c,2m_c,-m_c)
\overline q_a(\Phi_{AB})_b\varepsilon_cq_c,
\end{aligned}
\end{gathered}
\end{equation}
where each $\kappa$ is a positive constant, the number of slot assignments of the monomial in~\eqref{eq:vertex_variational_derivative}, all equal by total symmetry~\eqref{eq:vertex_symmetry}, divided by $2\pi(n-1)!$.
Since $\ksym(m_c)$, $r_c$ and the $\kappa$ are real, \eqref{eq:vertex_coefficient_form} reduces the theorem to the reality of the three contractions.

\emph{Direct coefficients.}
For $X_{\mathrm s,11}$, conjugation and the vertex identities of Lemma~\ref{lem:hamiltonian_vertex_identities} give
\[
\begin{aligned}
\overline{X_{\mathrm s,11}}
&=\sum_{a,b,c,d}T^{(4)}_{abcd}(m_c,-m_c,-m_c,m_c)\,q_a\overline q_b\overline q_cq_d\\
&=\sum_{a,b,c,d}T^{(4)}_{badc}(-m_c,m_c,m_c,-m_c)\,q_a\overline q_b\overline q_cq_d\\
&=\sum_{a,b,c,d}T^{(4)}_{abcd}(-m_c,m_c,m_c,-m_c)\,\overline q_aq_bq_c\overline q_d
=X_{\mathrm s,11}.
\end{aligned}
\]
Here the three equalities are vertex conjugation symmetry~\eqref{eq:vertex_conjugation}, vertex symmetry~\eqref{eq:vertex_symmetry} for the interchange of slots $1\leftrightarrow2$ and $3\leftrightarrow4$, and the relabeling $a\leftrightarrow b$, $c\leftrightarrow d$.
Thus $X_{\mathrm s,11}\in\mathbb R$, and~\eqref{eq:vertex_coefficient_form} gives $\Re\hat s_{1,1}=0$.

The same three steps give $X_{\mathrm s,12}\in\mathbb R$, so~\eqref{eq:vertex_coefficient_form} gives $\Re\hat s_{1,2}=0$.
Reflection parity~\eqref{eq:vertex_reflection} is not needed for either direct coefficient.

\emph{Mixed cascade coefficient.}
Let $V:=\mathcal F_x\bigl[[AB]\,\delta\mathsf H_3/\delta U(\psi_c)\bigr](2m_c)\in\mathbb C^2$.
By~\eqref{eq:quadratic_forcing_definition}, \eqref{eq:ham_variational_field} with~\eqref{eq:extraction_commutation}, and~\eqref{eq:K_symbol},
\[
\hat g_{AB}
=\mathcal F_x\bigl[[AB]\,G_2(\psi_c)\bigr](2m_c)
=\mathcal F_x\Bigl[S\mathcal K\,[AB]\,\frac{\delta\mathsf H_3}{\delta U}(\psi_c)\Bigr](2m_c)
=i\ksym(2m_c)\,SV.
\]
Substituting~\eqref{eq:center_linear_parametrization} into~\eqref{eq:vertex_variational_derivative} with $n=3$ and $\mathbf q_B=-R_0\overline{\mathbf q}$ gives
\[
V_a=-\frac{1}{2\pi}\sum_{b,c}
T^{(3)}_{abc}(-2m_c,m_c,m_c)q_b\varepsilon_c\overline q_c,
\qquad a=1,2.
\]
Componentwise conjugation and reflection parity~\eqref{eq:vertex_reflection} give
\[
\overline V_a
=-\frac{1}{2\pi}\sum_{b,c}T^{(3)}_{abc}(2m_c,-m_c,-m_c)
\overline q_b\varepsilon_cq_c
=-\frac{\varepsilon_a}{2\pi}\sum_{b,c}T^{(3)}_{abc}(-2m_c,m_c,m_c)
\varepsilon_b\overline q_bq_c = \varepsilon_aV_a
\]
by interchanging $2\leftrightarrow3$ and relabeling $b,c$.
Thus $\overline V=R_0V$.
Since $\beta_{AB}=0$, the second-harmonic coefficient is $\Phi_{AB}=(-M_{2m_c})^{-1}\hat g_{AB}$ by~\eqref{eq:mixed_second_harmonic}.
Inverting~\eqref{eq:symbol_parity}, legitimate because $\det M_{2m_c}>0$ by Assumption~\ref{ass:H}\ref{H1}, gives $\overline{M_{2m_c}^{-1}}=R_0M_{2m_c}^{-1}R_0$, and the conjugation law is preserved along
\begin{equation}\label{eq:vertex_Psi_conjugation}
\overline V=R_0V
\quad\Longrightarrow\quad
\overline{\hat g_{AB}}=R_0\hat g_{AB}
\quad\Longrightarrow\quad
\overline{\Phi_{AB}}=R_0\Phi_{AB}.
\end{equation}

For the outer contraction $X_{\mathrm c,12}$ in~\eqref{eq:vertex_cascade_contraction} the same three steps apply, with two changes: the conjugate of the middle leg is $\overline{(\Phi_{AB})_b}=\varepsilon_b(\Phi_{AB})_b$ by~\eqref{eq:vertex_Psi_conjugation}, and the parity factors $\varepsilon_b\varepsilon_c$ are absorbed by reflection parity~\eqref{eq:vertex_reflection}.
Hence $X_{\mathrm c,12}\in\mathbb R$, and~\eqref{eq:vertex_coefficient_form} gives $\Re\hat c_{1,2}=0$.

This proves~\ref{T1} and~\ref{T2}.
By the splitting~\eqref{eq:coefficient_splitting},
$\xiR=\Re(\hat s_{1,2}+\hat c_{1,2})=0$ and
$\zetaR=\Re(\hat s_{1,1}+\hat c_{1,1})=\Re\hat c_{1,1}$,
which proves the two identities of~\ref{T3}.
The last assertion follows from Proposition~\ref{prop:degeneracies}.

\emph{Wave selection.}
The cubic truncation of~\eqref{reduced_system} is the standard $O(2)$-Hopf normal form, whose equilibria and their stability are classified by the signs of $\mu=\Re\beta_{m_c,1}(\lambda)$, $\zetaR$ and $\zetaR\pm\xiR$~\cite{vangils1986hopf,crawford1991symmetry}.
See~\cite[Prop.~3.5]{ozersengultiryakioglu2026} in the present notation.
Since $\xiR=0$ and $\zetaR\ne0$ give $\zetaR\pm\xiR\ne0$, item~\ref{T4} is the case $\zetaR\gtrless0$ of that classification.

\begin{proof}[Proof of Proposition~\ref{prop:degeneracies}]
Variationality is not used here.
For~\ref{Dg1}, \eqref{eq:swap_cancellation} applied to~\eqref{eq:direct_coefficient_extraction}--\eqref{eq:cascade_coefficient_extraction} with $G_j=\mathcal KF_j$ shows that all four coefficients carry the factor $\ksym(m_c)$.
For~\ref{Dg2}, the forcings $\hat g_{AA}$ and $\hat g_{AB}$ of~\eqref{eq:quadratic_forcing_definition} are the mode-$2m_c$ coefficients of $\mathcal KF_2(\psi_c)$ and so carry the factor $i\ksym(2m_c)=2im_cP(4m_c^2)$, which vanishes when $P(4m_c^2)=0$.
The homological equations~\eqref{eq:self_second_harmonic}--\eqref{eq:mixed_second_harmonic} then give $\Phi_{AA}=\Phi_{AB}=0$, hence $\hat c_{1,1}=\hat c_{1,2}=0$.
\end{proof}

\begin{remark}\label{rem:surviving_coefficient}
The only potentially nonzero real contribution is $\Re\hat c_{1,1}$, the $AA$ cascade produced by $G_2$ acting twice.
Its center-manifold coefficient contains $(2\beta_1I-M_{2m_c})^{-1}$ rather than $(-M_{2m_c})^{-1}$.
Conjugation reverses the shift $2i\omega_c$, so the parity law~\eqref{eq:vertex_Psi_conjugation} for $\Phi_{AB}$ need not hold for $\Phi_{AA}$.
\end{remark}

\section{Derivative-free defect identities and proof of the converse}\label{sec:defects}

This section proves Theorem~\ref{thm:converse}.
Throughout this section $G_2$ and $G_3$ are derivative-free for a common multiplier $\mathcal K=\partial_xP(-\partial_x^2)$.

\begin{lemma}\label{lem:flux_uniqueness_parity}
Let $G$ satisfy Assumption~\ref{ass:S}\ref{S1}--\ref{S3} and let $G_j=\mathcal KF_j$ for a multiplier $\mathcal K$ with $F_j$ derivative-free.
Then $F_j=(F_{j,1},F_{j,2})^T$ is the unique derivative-free flux of $G_j$ for $\mathcal K$, satisfies
\begin{equation}\label{eq:flux_parity}
F_j(R_0U)=-R_0F_j(U),
\end{equation}
and has the form
\begin{equation}\label{eq:general_equivariant_flux}
\begin{aligned}
F_{2,1}&=g_{11}uv,
&F_{2,2}&=\tfrac12(g_{21}u^2+g_{22}v^2),\\
F_{3,1}&=h_{11}u^2v+h_{12}v^3,
&F_{3,2}&=h_{21}u^3+h_{22}uv^2,
\end{aligned}
\end{equation}
with real coefficients.
\end{lemma}

\begin{proof}
Suppose $F$ is a homogeneous polynomial map of degree $j$ with $\mathcal KF(U)=0$ for all $U$.
Take $U=\mathbf w\cos kx$ with constant $\mathbf w\in\mathbb R^2$.
Then, by~\eqref{eq:K_symbol},
\[
F(U)=F(\mathbf w)\cos^jkx,
\qquad
\widehat{\mathcal KF(U)}(jk)=i\,jk\,P(j^2k^2)\,2^{-j}F(\mathbf w),
\]
and $jkP(j^2k^2)\ne0$ for all large $k$, since $P$ has finitely many roots, forcing $F(\mathbf w)=0$ for every $\mathbf w$.
Applied to the difference of two derivative-free fluxes, this proves uniqueness of $F_j$.
For~\eqref{eq:flux_parity}, $\mathcal KR=-R\mathcal K$, since $\mathcal K$ has only odd derivatives, and $RG_jR=G_j$ by~\ref{S2} give
\[
\mathcal K\bigl(F_j+RF_jR\bigr)
=G_j-RG_jR
=0,
\]
hence $F_j+RF_jR=0$ by uniqueness, which is~\eqref{eq:flux_parity} pointwise.
Since $R_0=\operatorname{diag}(1,-1)$, \eqref{eq:flux_parity} makes $F_{j,1}$ odd and $F_{j,2}$ even in $v$, and homogeneity of degree $j$ leaves the monomials in~\eqref{eq:general_equivariant_flux}.
\end{proof}

The \emph{quadratic} and \emph{cubic variational defects} of the flux~\eqref{eq:general_equivariant_flux} are
\[
\mathcal D_2=g_{22}-g_{11},
\qquad
\mathcal D_3=h_{11}-h_{22}.
\]

\begin{lemma}\label{lem:defects_variational}
Under the hypotheses of Lemma~\ref{lem:flux_uniqueness_parity}, $\mathcal D_j=0$ if and only if $F_j$ is variational, $j=2,3$.
\end{lemma}

\begin{proof}
For the form~\eqref{eq:general_equivariant_flux}, the Helmholtz condition
\begin{equation}\label{eq:hamiltonian_flux}
\partial_uF_{j,1}=\partial_vF_{j,2}
\end{equation}
is equivalent to $\mathcal D_j=0$.
Since $S$ is the swap matrix, \eqref{eq:hamiltonian_flux} says that $SF_j$ is a gradient, and Euler's identity $U\cdot\nabla_UH=(j+1)H$ for a homogeneous polynomial $H$ of degree $j+1$ gives its potential as
\begin{equation}\label{eq:derivative_free_potential}
H_{j+1}(U)=\frac{1}{j+1}\,U\mathbin{\cdot}SF_j(U),
\qquad
F_j=S\nabla_UH_{j+1}.
\end{equation}
\emph{$\mathcal D_j=0$ implies variationality.}
If $\mathcal D_j=0$, the derivative-free density $H_{j+1}$ defines a local Hamiltonian $\mathsf H_{j+1}$ with $\delta\mathsf H_{j+1}/\delta U=\nabla_UH_{j+1}$, which is $R$-invariant because $H_{j+1}(R_0U)=H_{j+1}(U)$ by~\eqref{eq:flux_parity} and $SR_0=-R_0S$.
Hence $F_j$ is variational.
\emph{Variationality implies $\mathcal D_j=0$.}
If $F_j=S\,\delta\mathsf H_{j+1}/\delta U$ for an $R$-invariant local Hamiltonian $\mathsf H_{j+1}$, then homogeneity and~\eqref{eq:variational_derivative_definition} give
\[
\mathsf H_{j+1}(U)=\frac{1}{j+1}\int_{\mathbb T}U\mathbin{\cdot}\frac{\delta\mathsf H_{j+1}}{\delta U}\,dx
=\int_{\mathbb T}H_{j+1}(U)\,dx
\]
with $H_{j+1}$ as in~\eqref{eq:derivative_free_potential}, so $SF_j=\nabla_UH_{j+1}$ is a gradient, \eqref{eq:hamiltonian_flux} holds, and $\mathcal D_j=0$.
\end{proof}

Since $\det M_{2m_c}(\lambda_c)>0$ the symbol is invertible at mode $2m_c$, and parity gives its inverse the form
\begin{equation}\label{eq:inverse_symbol_parity}
M_{2m_c}^{-1}
=\begin{pmatrix}\varrho_{11}&i\varrho_{12}\\
i\varrho_{21}&\varrho_{22}\end{pmatrix},
\qquad \varrho_{ij}\in\mathbb R.
\end{equation}
Abbreviate
\begin{equation}\label{eq:E_of_L}
E=\varrho_{22}\bigl(g_{22}\abs{v_c}^2-g_{21}\abs{u_c}^2\bigr)-2\varrho_{21}g_{11}\chi,
\end{equation}
with $\chi=\Im(u_c\overline{v_c})$ as in~\eqref{eq:qc}.

\begin{proposition}[Helmholtz-defect factorization]
\label{prop:ham_explicit_defects}
Suppose $(L_\lambda,G)$ is admissible and $G_2$, $G_3$ are derivative-free for a common multiplier $\mathcal K$.
Then
\begin{align}
\Re\hat c_{1,2}
={}&\frac{\ksym(m_c)\ksym(2m_c)}2\,\mathcal D_2\,E,
\label{eq:ham_quadratic_identity}\\
\Re\hat s_{1,1}
={}&-\ksym(m_c)\chi\,\mathcal D_3,
\label{eq:ham_cubic_self_identity}\\
\Re\hat s_{1,2}
={}&2\ksym(m_c)\chi\,\mathcal D_3.
\label{eq:ham_cubic_identity}
\end{align}
\end{proposition}

\begin{proof}
Evaluating the extraction formulas~\eqref{eq:direct_coefficient_extraction}--\eqref{eq:cascade_coefficient_extraction} on the flux~\eqref{eq:general_equivariant_flux} and solving the homological equations~\eqref{eq:self_second_harmonic}--\eqref{eq:mixed_second_harmonic} at mode $2m_c$ gives the three identities.
The calculation is in Appendix~\ref{app:computation}.
\end{proof}

Since variational fluxes have $\mathcal D_2=\mathcal D_3=0$ by Lemma~\ref{lem:defects_variational}, Proposition~\ref{prop:ham_explicit_defects} implies Theorem~\ref{thm:hamiltonian_flux} in the derivative-free case.

The invariant $\chi$ vanishes exactly when $M_{m_c}$ has vanishing diagonal, equivalently when $\mathbf q$ is a complex multiple of a real vector.
In that case $\mathcal D_3$ is invisible to $\zetaR$ and $\xiR$.

If $G_2$, $G_3$ are not both variational, there may still exist an admissible $L_\lambda$ with $\xiR=\Re\hat c_{1,2}+\Re\hat s_{1,2}=0$.
We now construct a family of admissible linear parts such that $\xiR=0$ on the whole family forces $G_2$ and $G_3$ to be variational.

\begin{lemma}\label{lem:witness}
Let $G$ satisfy Assumption~\ref{ass:S}\ref{S1}--\ref{S3}, fix $m_c\ge1$ and an even integer $N$ with $2N-1\ge r_G$, and let $r,\nu>0$ and $\abs a<1$.
Then
\[
L_\lambda
=\begin{pmatrix}\tau_\lambda+a&r\partial_x\\ r^{-1}\partial_x&\tau_\lambda-a\end{pmatrix},
\qquad
\tau_\lambda=\lambda-\nu(-\partial_x^2-m_c^2)^N,
\]
is admissible for $G$ with critical wavenumber $m_c$ and $\lambda_c=0$, and, with the critical eigenvector normalized by $v_c=1$, its invariants are
\begin{equation}\label{eq:witness_invariants}
\chi=-\frac{ar}{m_c},
\qquad
\abs{u_c}^2=r^2,
\qquad
\abs{v_c}^2=1,
\qquad
\varrho_{21}=-\frac{2m_c/r}{d^2+4m_c^2-a^2},
\qquad
\varrho_{22}=\frac{a-d}{d^2+4m_c^2-a^2},
\end{equation}
with $d=\nu(3m_c^2)^N>0$.
\end{lemma}

The verification of Assumptions~\ref{ass:H} and~\ref{ass:D} and the computation of the invariants are in Appendix~\ref{app:converse}.
The lemma proves Proposition~\ref{prop:witness}.

\begin{proof}[Proof of Theorem~\ref{thm:converse}]
That variationality forces uniform cancellation is Theorem~\ref{thm:hamiltonian_flux}.
For the converse implication, write $\ksym_1=\ksym(m_c)=m_cP(m_c^2)$ and $\ksym_2=\ksym(2m_c)=2m_cP(4m_c^2)$.
The hypothesis $P(m_c^2)P(4m_c^2)\ne0$ gives $\ksym_1\ksym_2\ne0$.
Adding~\eqref{eq:ham_quadratic_identity} and~\eqref{eq:ham_cubic_identity} gives, for every admissible $L_\lambda$,
\begin{equation}\label{eq:converse_master}
\xiR
=2\ksym_1\chi\,\mathcal D_3
+\frac{\ksym_1\ksym_2}2\,\mathcal D_2\,E,
\end{equation}
with $E$ as in~\eqref{eq:E_of_L}.
For a fixed flux, only $\chi$, $\varrho_{21}$, $\varrho_{22}$, $u_c$, $v_c$ vary with $L$.
By assumption $\xiR=0$ for the witness family $L_{a,r,\nu}$ of Lemma~\ref{lem:witness}, with invariants~\eqref{eq:witness_invariants}.
The proof uses two features: $\chi$ vanishes exactly when $a=0$, and it is independent of $\nu$, whereas $\varrho_{21}$ is strictly monotone in $\nu$.

Suppose first $(g_{21},g_{22})\ne(0,0)$.
Set $a=0$ and choose $r>0$ so that $g_{22}-g_{21}r^2\ne0$.
Then $\chi=0$ and $E=\varrho_{22}(g_{22}-g_{21}r^2)\ne0$ since $\varrho_{22}=-d/(d^2+4m_c^2)\ne0$, so~\eqref{eq:converse_master} and $\ksym_1\ksym_2\ne0$ give $\mathcal D_2=0$.
Taking any $0<\abs a<1$ next gives $\chi\ne0$, so $\mathcal D_3=0$.

Suppose instead $(g_{21},g_{22})=(0,0)$, so $\mathcal D_2=-g_{11}$ and $E=-2\varrho_{21}g_{11}\chi$.
Fix $r$ and $0<\abs a<1$ and vary $\nu$.
Then $\chi\ne0$ is fixed, while~\eqref{eq:converse_master} gives
\[
0=\xiR
=\ksym_1\chi\bigl[2\mathcal D_3-\ksym_2\varrho_{21}g_{11}\mathcal D_2\bigr],
\qquad\text{hence}\qquad
2\mathcal D_3=\ksym_2\varrho_{21}g_{11}\mathcal D_2 .
\]
Since $\varrho_{21}$ is nonconstant in $\nu$, two values give $g_{11}\mathcal D_2=-g_{11}^2=0$.
Thus $\mathcal D_2=0$ and then $\mathcal D_3=0$.

For the last assertion, suppose~\eqref{eq:multiplier_nondegeneracy} fails.
If $P(m_c^2)=0$, every cubic coefficient vanishes by Proposition~\ref{prop:degeneracies}\ref{Dg1}, so $\xiR=0$ for every admissible $L_\lambda$ whatever $G_2$, $G_3$ are, in particular for non-variational ones.
If $P(m_c^2)\ne0$ but $P(4m_c^2)=0$, then $\xiR=\Re\hat s_{1,2}$ by Proposition~\ref{prop:degeneracies}\ref{Dg2}, which by~\eqref{eq:direct_coefficient_extraction} is independent of $F_2$ and vanishes for variational $F_3$ by~\eqref{eq:ham_cubic_identity}.
So $\xiR=0$ for every admissible $L_\lambda$ when $G_3$ is variational and $G_2$ is not.
\end{proof}

\section{Examples and a counterexample}\label{sec:scope}

\begin{example}[An explicit derivative-dependent family]\label{ex:derivative_family}
Fix $m_c\ge1$, $\nu>0$, $\gamma\ne0$, real coefficients $a_1,\ldots,a_4$, and let
\begin{equation}\label{eq:derivative_model}
L_\lambda
=\bigl[\lambda-\nu(-\partial_x^2-m_c^2)^2\bigr]I
+\gamma S\partial_x,
\qquad
G=S\partial_x\left(
\frac{\delta\mathsf H_3}{\delta U}
+\frac{\delta\mathsf H_4}{\delta U}
\right),
\end{equation}
where $\mathcal K=\partial_x$,
\begin{equation}\label{eq:derivative_model_hamiltonians}
\mathsf H_3(U)
=\int_{\mathbb T}\left(
\frac{a_1}{6}u^3+\frac{a_2}{2}uv^2
+\frac{a_3}{2}u u_x^2+\frac{a_4}{2}u v_x^2
\right)\,dx,
\end{equation}
and $\mathsf H_4$ is any $R$-invariant quartic local Hamiltonian whose density involves at most first derivatives.
Then:
\begin{enumerate}[label=(\roman*)]
\item $(L_\lambda,G)$ is admissible with critical wavenumber $m_c$ and $\lambda_c=0$, and $G$ is variational for $\mathcal K=\partial_x$;
\item independently of $\mathsf H_4$,
\begin{align*}
(a_1,a_2,a_3,a_4)&=(1-3m_c^2,1,1,0)
&\Longrightarrow\quad
\zetaR&<0,\\
(a_1,a_2,a_3,a_4)&=(0,1,-m_c^{-2},0)
&\Longrightarrow\quad
\zetaR&>0,
\end{align*}
so the standing wave is supercritical and stable, respectively subcritical and unstable.
\end{enumerate}
\end{example}

\begin{proof}
The symbol of $L_\lambda$ and its eigenvalues are
\[
M_m=\tau_m I+i\gamma mS,
\qquad
\tau_m=\lambda-\nu(m^2-m_c^2)^2,
\qquad
\beta_{m,\pm}=\tau_m\pm i\gamma m.
\]
For~(i), the displayed eigenvalues verify Assumption~\ref{ass:H} at $\lambda_c=0$ with $\omega_c=\abs{\gamma}m_c$, the determinant condition of~\ref{H1} holding because $\det M_m=\tau_m^2+\gamma^2m^2>0$ for $m\ne0$.
The principal part $-\nu\partial_x^4I$ is diagonal and nondegenerate while the off-diagonal term has order one and the parameter enters at order zero, so Assumption~\ref{ass:D} holds with $m_L=2$.
Finally, $\mathsf H_3$, $\mathsf H_4$ are $R$-invariant with $r_G\le3=2m_L-1$, so Assumption~\ref{ass:S} holds.
For~(ii), the extraction formulas of~\S\ref{subsec:extraction}, with critical eigenvector $\mathbf q=(\operatorname{sgn}\gamma,1)^T$, give
\[
\zetaR=-\frac{4}{9\nu m_c^2}
\qquad\text{and}\qquad
\zetaR=\frac{36\nu m_c^4}{16\gamma^2+81\nu^2m_c^6}.
\]
\end{proof}

The above family has a physical interpretation: for $a_2=a_4=0$, $\mathsf H_3$ is the cubic part of the capillary energy $\int_{\mathbb T}\bigl(\tfrac12v^2+W(u)+\tfrac12K(u)u_x^2\bigr)\,dx$, so $a_3=K'(0)$ measures state-dependent capillarity.

\begin{example}[$p$-system]\label{ex:p_system}
The fourth-order diffusive regularization studied in~\cite{liyao2015,yao20142},
\begin{equation}\label{eq:regularized_p_system}
\partial_tu-\partial_xv=-a\partial_x^4u,
\qquad
\partial_tv-\partial_x\sigma(u)=-\lambda\partial_x^2v-\partial_x^4v,
\end{equation}
has the nonlinearity of the inviscid $p$-system $u_t=v_x$, $v_t=\sigma(u)_x$.
Let $\sigma(u)=c^2u+\tfrac12u^2+\tfrac{\eta_p}3u^3$ with $c>0$, and let $\lambda$ be the bifurcation parameter.
Then:
\begin{enumerate}[label=(\roman*)]
\item $G_2$ and $G_3$ are variational for $\mathcal K=\partial_x$;
\item the pair is admissible with $m_c=1$ for $0<a<c$ and $c^2>4a(a+1)^3/27$;
\item $\xiR=0$, and $\zetaR$ is independent of $\eta_p$.
\end{enumerate}
\end{example}

\begin{proof}
For~(i), with $W'=\sigma$ and $H=\tfrac12v^2+W(u)$ the $p$-system reads $U_t=S\partial_x\nabla_UH$, so
\[
G_j=S\partial_x\,\frac{\delta\mathsf H_{j+1}}{\delta U},
\qquad
\mathsf H_3=\frac16\int_{\mathbb T}u^3\,dx,
\qquad
\mathsf H_4=\frac{\eta_p}{12}\int_{\mathbb T}u^4\,dx .
\]
Item~(ii) is~\cite[\S\S5.1, 6.1]{ozersengultiryakioglu2026}, and~(iii) is Theorem~\ref{thm:hamiltonian_flux}\ref{T3}.
\end{proof}

Item~(iii) comprises the two observations of~\cite{yao20142,liyao2015} recalled in the introduction.
See also~\cite[Eq.~(45)]{ozersengultiryakioglu2026}.

The third model is also from the literature, but its linear part fails Assumption~\ref{ass:D} and must be augmented before the theorem applies.
Section~\ref{sec:boussinesq} studies it near its Hamiltonian limit.

\begin{example}[A Boussinesq flux under a dissipative regularization]\label{ex:boussinesq}
Consider, on the $2\pi$-periodic mean-zero domain,
\[
v_t=p_x-a\,v_{xxxx},
\qquad
p_t=\partial_x\bigl(c_0^2v-\tfrac\alpha2v^2-\beta v_{xx}\bigr)-\alpha_2p_{xx}-\alpha_4p_{xxxx},
\]
with $\alpha\ne0$, $\alpha_4>0$, $c_0^2>0$, $\beta\ge0$ and $a\ge0$.
The case $a=0$ is the Boussinesq model of Christov, Maugin and Porubov~\cite{christovmauginporubov2007}.
Take $\lambda=\alpha_2$ and write $L^a_\lambda$ for the linear part.
Then:
\begin{enumerate}[label=(\roman*)]
\item $G=G_2$ is variational for $\mathcal K=\partial_x$;
\item $(L^0_\lambda,G)$ is not admissible for any $\beta\ge0$;
\item for $0<a<\min\bigl\{3\alpha_4,\sqrt{c_0^2+\beta}\bigr\}$, the pair $(L^a_\lambda,G)$ is admissible with $m_c=1$ and $\lambda_c=a+\alpha_4$ for every $\beta\ge0$.
\end{enumerate}
\end{example}

\begin{proof}
For~(i), $G=S\partial_x\,\delta\mathsf H_3/\delta U$ with $\mathsf H_3=-\tfrac\alpha6\int_{\mathbb T}v^3\,dx$.
For~(ii) and~(iii), the symbol of $L^a_\lambda$ is
\[
M_m=\begin{pmatrix}-am^4&im\\ im(c_0^2+\beta m^2)&\alpha_2m^2-\alpha_4m^4\end{pmatrix},
\]
so at $\lambda=a+\alpha_4$
\[
\tr M_m=m^2(a+\alpha_4)(1-m^2),
\qquad
\det M_m=m^2\bigl[c_0^2+\beta m^2+am^4\bigl(\alpha_4(m^2-1)-a\bigr)\bigr],
\]
and $\partial_\lambda\tr M_1=1$.
For~(ii), at $a=0$ the $(1,1)$ entry vanishes, so~\ref{D} fails.
The trace vanishes only at $m=\pm1$, $\det M_m=m^2(c_0^2+\beta m^2)>0$, and $\det M_m/\abs{\tr M_m}\to\beta/\alpha_4$, so $L^0_\lambda$ satisfies Assumption~\ref{ass:H} with $m_c=1$ and $\lambda_c=\alpha_4$ exactly when $\beta>0$, the spectral gap~\ref{H4} closing at $\beta=0$.
For~(iii), the trace is negative off $m=\pm1$, the determinant is positive at $m=\pm1$ by $a^2<c_0^2+\beta$ and for $\abs m\ge2$ by $a<3\alpha_4$, and $\det M_m/\abs{\tr M_m}\to\infty$, so Assumption~\ref{ass:H} holds.
Assumption~\ref{ass:D} holds with $m_L=2$, the principal part $\operatorname{diag}(-a,-\alpha_4)\partial_x^4$ being nondegenerate with the off-diagonal terms of order three and $\lambda$ entering at order two, and~\ref{S4} holds with $r_G=1$.
\end{proof}

We do not claim that the regularization $-a\,v_{xxxx}$ has any physical meaning.
Its precedent is the diffusion coefficient $a$ of~\cite{yao20142}.

The last example is a non-variational flux that does not alter the cubic normal-form coefficients, so Theorem~\ref{thm:converse} does not extend to derivative-dependent fluxes.

\begin{example}[A derivative-dependent null form]\label{ex:null_form}
Let
\[
F_2(U)=\bigl(0,\mathcal N(u,u)\bigr)^T,
\qquad
F_3=0,
\qquad
\mathcal N(f,g)=2\partial_x^4(fg)-7\partial_x^2(f_xg_x)-4f_{xx}g_{xx}.
\]
Then:
\begin{enumerate}[label=(\roman*)]
\item $G_2=\mathcal KF_2$ satisfies Assumption~\ref{ass:S}\ref{S1}--\ref{S3} with $r_G=\operatorname{ord}\mathcal K+4$;
\item $G_2$ is not variational;
\item for all fluxes $F_2'$, $F_3'$ and every $L_\lambda$ admissible for both, the nonlinearities $\mathcal KF_2+\mathcal KF_2'+\mathcal KF_3'$ and $\mathcal KF_2'+\mathcal KF_3'$ have the same cubic normal-form coefficients~\eqref{eq:direct_coefficient_extraction}--\eqref{eq:cascade_coefficient_extraction}, and in particular, for $F_2'=F_3'=0$ all four vanish.
\end{enumerate}
\end{example}

\begin{proof}
For~(i), $\mathcal KF_2$ is $R$-equivariant since every term of $\mathcal N$ has even total differential order, and Proposition~\ref{prop:witness} supplies admissible operators.
For~(iii), the symbol $b$ of $\mathcal N$, defined by $\mathcal N(e^{ik_1x},e^{ik_2x})=b(k_1,k_2)e^{i(k_1+k_2)x}$, is
\[
b(k_1,k_2)=(k_1-k_2)^2(2k_1+k_2)(k_1+2k_2),
\qquad\text{so}\qquad
b(\pm m_c,\pm m_c)=b(\mp m_c,\pm2m_c)=0 ,
\]
while $(\pm m_c,\mp m_c)$ has output wavenumber $0$, annihilated by $\mathcal K$ by~\eqref{eq:K_symbol}.
Hence $\hat g_{AA}=\hat g_{AB}=0$, both being multiples of $b(m_c,m_c)$ since the $A$ and $B$ legs of~\eqref{eq:center_linear_parametrization} both carry wavenumber $m_c$.
So the center-manifold term of $\mathcal KF_2$ alone vanishes and, with $F_3=0$, so do its four coefficients.
Let $h_2'$ be the center-manifold term of $\mathcal KF_2'$.
The forcing~\eqref{eq:quadratic_forcing_definition} is linear in $G_2$, so for each monomial $XY$
\[
\hat g_{XY}(\mathcal KF_2'+\mathcal KF_2)=\hat g_{XY}(\mathcal KF_2')+\hat g_{XY}(\mathcal KF_2)=\hat g_{XY}(\mathcal KF_2'),
\qquad\text{hence}\qquad
h_2=h_2'
\]
by~\eqref{eq:quadratic_homological_equation}, and, for $n=1,2$,
\[
\mathcal P_n\bigl(\mathcal K\,DF_2(\psi_c)\,h_2'\bigr)=0,
\]
since $h_2'$ is supported on modes $\pm2m_c$ by~\eqref{eq:zero_mode_vanishing}, the mode-$\pm m_c$ output of $DF_2(\psi_c)h_2'$ comes from the interactions $(\mp m_c,\pm2m_c)$ with $b(\mp m_c,\pm2m_c)=0$, and the interactions $(\pm m_c,\pm2m_c)$ produce only modes $\pm3m_c$, which the critical projections annihilate.
The direct coefficients depend on $F_3'$ alone.
For~(ii), $F_2$ is the only differential-polynomial flux of $G_2$: a quadratic differential polynomial $Q$ with $\mathcal KQ(U)=0$ for all $U$ has a polynomial symbol vanishing off the finitely many lines $k_1+k_2=m$ with $\ksym(m)=0$, so $Q=0$.
If $SF_2=(\mathcal N(u,u),0)^T=\delta\mathsf H_3/\delta U$ for a cubic $\mathsf H_3$ as in~\eqref{eq:local_hamiltonian}, then $\mathsf H_3$ depends on $u$ alone and $T(f,g,h)=\int_{\mathbb T}h\,\mathcal N(f,g)\,dx=\tfrac12D^3\mathsf H_3(f,g,h)$ is symmetric.
Since $T(e^{ik_1x},e^{ik_2x},e^{ik_3x})=2\pi\,b(k_1,k_2)$ for $k_1+k_2+k_3=0$, this forces $b(k_1,k_2)=b(k_2,k_3)$, whereas
\[
b(1,2)=20\ne-100=b(2,-3).
\qedhere
\]
\end{proof}

Whether a lower-order (with $r_G<\operatorname{ord}\mathcal K+4$) non-variational flux with the same property exists is part of Problem~\ref{prob:kernel}.

\section{The cubic coefficient near the Hamiltonian limit}\label{sec:boussinesq}

In this section we regard~\eqref{main} as a dissipative linear perturbation of a Hamiltonian system and follow $\zetaR$ as the perturbation is removed.
By Theorem~\ref{thm:hamiltonian_flux} it is the only possibly non-vanishing real cubic coefficient, and by Remark~\ref{rem:hamiltonian_limit} it vanishes in the nonresonant Hamiltonian limit, so it is generated by the dissipation alone.
The question is how it scales with the dissipation.
We call $\abs{\zetaR}$ the \emph{saturation strength}, since the bifurcating amplitudes are $\sqrt{\abs\mu/\abs{\zetaR}}$ to leading order in the coordinates of~\eqref{reduced_system}.

\begin{definition}[Perturbed-Hamiltonian family]\label{def:perturbed_hamiltonian}
A \emph{perturbed-Hamiltonian family} is a two-parameter family $L_{\lambda,\epsilon}=H+D_{\lambda,\epsilon}$, $\epsilon>0$, such that
\begin{enumerate}[label=(\alph*)]
\item $HU=S\mathcal K\,\delta\mathsf H_2/\delta U$ for a quadratic $\mathsf H_2$ as in~\eqref{eq:local_hamiltonian}, and $G_2$, $G_3$ are variational for the same $\mathcal K$;
\item $(L_{\cdot,\epsilon},G)$ is admissible for every $\epsilon>0$, with critical parameter $\lambda_c(\epsilon)$;
\item $D_{\lambda_c(\epsilon),\epsilon}\to0$ coefficientwise as $\epsilon\to0$.
\end{enumerate}
Write $\pm i\omega_m$ for the eigenvalues of the symbol of $H$ at mode $m$ when they are purely imaginary, so that $\omega_{m_c}=\lim_{\epsilon\to0}\omega_c$.
The \emph{detuning} of the doubled mode is
\begin{equation}\label{eq:detuning}
\delta:=2\omega_{m_c}-\omega_{2m_c}.
\end{equation}
\end{definition}

\begin{remark}\label{rem:hamiltonian_limit}
As $\epsilon\to0$ the spectral gap closes and Proposition~\ref{thm:reduction} no longer applies.
Nevertheless, for $L=H$ with $\omega_{m_c}>0$ and $\omega_{2m_c}\notin\{0,2\omega_{m_c}\}$ the resolvents in~\eqref{eq:self_second_harmonic}--\eqref{eq:mixed_second_harmonic} exist, so the extraction formulas~\eqref{eq:direct_coefficient_extraction}--\eqref{eq:cascade_coefficient_extraction} still apply.
Since the critical eigenvalues $\pm i\omega_{m_c}$ are simple, the critical space is symplectic, so the formal cubic dynamics is the Hamiltonian $O(2)$ normal form~\cite[\S2, Eqs.~(2.2) and~(2.6)]{knoblochmahalovmarsden1994}, see also~\cite{meyerhalloffin2009,chossat2002hamiltonian}, in which the amplitude moduli are conserved: $\zetaR=\xiR=0$.
\end{remark}

\begin{proposition}[The Boussinesq family near the Hamiltonian limit]
\label{prop:boussinesq_ray}
The augmented Boussinesq pair of Example~\ref{ex:boussinesq}(iii) forms a perturbed-Hamiltonian family as $\epsilon:=a+\alpha_4\to0$, with $\mathsf H_2=\int_{\mathbb T}\bigl(\tfrac12p^2+\tfrac{c_0^2}2v^2+\tfrac\beta2v_x^2\bigr)dx$, $\mathsf H_3=-\tfrac\alpha6\int_{\mathbb T}v^3\,dx$ and $\omega_m=m\sqrt{c_0^2+\beta m^2}$, so its doubled mode is $2{:}1$ resonant exactly when $\beta=0$.
Let $(L^a_\lambda,G)$ be that pair at onset, $\alpha_2=\lambda_c=\epsilon$, with the critical eigenvector normalized by its second component.
For every admissible $\epsilon>0$, $\zetaR<0$, so the standing wave is the selected branch for every $\beta\ge0$.
Along any such path with $\epsilon\to0$, with constants $C_\beta,C_0>0$ depending on $\alpha,c_0,\beta$ only:
\begin{enumerate}[label=(\roman*)]
\item\label{B1} For $\beta>0$ the detuning~\eqref{eq:detuning} is nonzero and
\[
\zetaR=-C_\beta\,\epsilon+O(\epsilon^3).
\]
\item\label{B2} For $\beta=0$ the detuning vanishes and
\[
\zetaR=-C_0\,\epsilon^{-1}+O(\epsilon).
\]
\end{enumerate}
\end{proposition}

The proof, a direct cascade evaluation ending in a closed form for $\zetaR$, is in Appendix~\ref{app:boussinesq}.

\begin{remark}\label{rem:precedents}\leavevmode
\begin{enumerate}[label=(\alph*)]
\item\label{R0}
The detuning vanishes exactly at the $2{:}1$ resonance.
There the shifted resolvent $(2i\omega_cI-M_{2m_c})^{-1}$ of~\eqref{eq:self_second_harmonic}, which exists for every $\epsilon>0$ by Assumption~\ref{ass:H}, is unbounded as $\epsilon\to0$, and this is the source of the divergence in item~\ref{B2} of Proposition~\ref{prop:boussinesq_ray}.
\item\label{R1}
Items~\ref{B1} and~\ref{B2} of Proposition~\ref{prop:boussinesq_ray} are limits along the lines $\beta=\mathrm{const}$ of the quadrant $\epsilon>0$, $\beta\ge0$.
The joint limit $(\beta,\epsilon)\to(0,0)$ is also of interest, since dispersion and damping then compete: by~\eqref{eq:boussinesq_zeta}, $\zetaR\sim-\alpha^2\epsilon/(6c_\beta^2\beta^2)$ when $\beta\gg\epsilon$ and $\zetaR\sim-\alpha^2/(24c_\beta^4\epsilon)$ when $\beta\ll\epsilon$, and the two regimes meet where $\beta$ and $\epsilon$ are comparable.
We do not pursue that limit here (Problem~\ref{prob:uniformity} of Section~\ref{sec:open}).
\item\label{R2}
Resonance does not by itself force the divergence: the shifted resolvent is unbounded only along the eigenvector of $M_{2m_c}$ whose eigenvalue tends to $2i\omega_{m_c}$, and whether the quadratic forcing has a component along it depends on the nonlinearity.
The linear part of Example~\ref{ex:derivative_family} is $2{:}1$ resonant as $\nu\to0$, since $\omega_m=\abs\gamma m$, and its two choices of $\mathsf H_3$ give $\zetaR\propto\nu^{-1}$ and $\zetaR\propto\nu$.
\item\label{R3}
The resonance $\omega(2k)=2\omega(k)$ at $\beta=0$ is that of Wilton's ripples~\cite{mcgoldrick1970}.
The same dichotomy, real cubic coefficients of the order of the damping off resonance and of its inverse at a $2{:}1$ resonance, occurs in weakly damped Faraday waves~\cite{portersilber2002,topazsilber2002}, where a parametric drive couples the two amplitudes.
Here the family is autonomous and $\xiR=0$ throughout.
\end{enumerate}
\end{remark}

\section{Limitations and open problems}\label{sec:open}

The restrictions of Sections~\ref{sec:setting} and~\ref{sec:main} leave three principal questions and several further directions.
\begin{enumerate}[label=\arabic*.,ref=\arabic*]
\item\label{prob:uniformity}
Section~\ref{sec:boussinesq} follows $\zetaR$ to $\epsilon\to0$, but the reduction that gives it meaning holds in a neighborhood that may shrink with $\epsilon$: at the resonance the eliminated doubled mode is damped only at rate $O(\epsilon)$~\cite{rucklidgesilber2009}, and keeping it as a second critical mode~\cite{knoblochproctor1988,leblanclangford1996} does not help at $\beta=0$, where every harmonic is resonant~\cite{schneider2005}.
Is there a reduction uniform in $\epsilon$, at fixed $\beta$ and in the joint limit of Remark~\ref{rem:precedents}\ref{R1}, and does it reproduce the two laws of Proposition~\ref{prop:boussinesq_ray}?

\item
When $\zetaR=0$ for a variational flux, as happens whenever $G_2=0$, both cubic real parts vanish and selection is decided at fifth order or beyond~\cite{crawford1988degenerate,villarsepulveda2024amplitude}.
Which real parts of the fifth-order coefficients vanish when the nonlinearity is Hamiltonian through fifth order?

\item\label{prob:kernel}
Identify all conservative reflection-compatible quadratic and cubic fluxes whose addition to any flux alters neither $\zeta$ nor $\xi$ for any admissible $L_\lambda$.
\end{enumerate}

Several further directions change the setting rather than the question.
\begin{enumerate}[label=(\alph*)]
\item
For $d>2$ the adjoint eigenvector is no longer fixed by reflection alone, and for $d\ge4$ pure damping can destabilize a spectrally stable Hamiltonian equilibrium with indefinite energy~\cite{bkmr1994}.

\item
For Lie--Poisson operators such as that of the Green--Naghdi hierarchy~\cite{matsuno2015hamiltonian}, the linear part of the operator's expansion acting on $\nabla\mathsf H_2$ contributes to $G_2$, and the variational hypothesis must be reformulated.

\item
For nonlocal multipliers whose Fourier symbol is $i\ksym(m)$ with $\ksym$ real and odd, the vertex argument extends unchanged, while the reduction needs Fourier-multiplier hypotheses.
Motivating cases are water waves~\cite{bambusi2021hamiltonian,berti2024hamiltonian}, the Hamiltonian bidirectional Whitham system of~\cite{dinvaydutykhkalisch2019} and the Hamiltonian $abcd$-Boussinesq systems~\cite{kwakmunozpobletepozo2019}, whose Poisson operators are $S\mathcal K$ with $\mathcal K=-\partial_x\mathcal K_0$, $\widehat{\mathcal K_0}(k)=\tanh(Hk)/(Hk)$, respectively $\mathcal K=-(1-\partial_x^2)^{-1}\partial_x$.

\item
Without a spectral gap the center-manifold reduction used here no longer applies and alternative methods need additional hypotheses, as for viscous shocks~\cite{poganyaozumbrun2015} or Cahn--Hilliard fronts~\cite{gohhosek2024}.

\item
With periodic coefficients the homological equations move to Bloch sectors, and with $D_k$ in place of $O(2)$ the normal form may contain additional cubic terms.
\end{enumerate}
Which of these preserve the cancellation, and under what hypotheses?

\section*{Code availability}

The cubic coefficients of~\eqref{reduced_system} can be evaluated with the \texttt{o2sym} package, release~v0.4.1, \url{https://doi.org/10.5281/zenodo.22096063}.
Its test suite samples admissible linear operators and derivative-dependent densities and verifies the defect identities of Proposition~\ref{prop:ham_explicit_defects} numerically.
The package implements the extraction formulas of~\cite{ozersengultiryakioglu2026} rather than the arguments of this paper, so these are independent checks of the identities.
Four short scripts, supplied as supplementary material, check the coefficient formulas of Sections~\ref{sec:scope} and~\ref{sec:boussinesq} and the identities of Proposition~\ref{prop:ham_explicit_defects} directly from the extraction formulas of~\S\ref{subsec:extraction} and the homological equations~\eqref{eq:quadratic_homological_equation}.
\texttt{appendix\_defects\_sym.py} rederives the three identities of Proposition~\ref{prop:ham_explicit_defects} symbolically from~\eqref{eq:general_equivariant_flux} and~\eqref{eq:inverse_symbol_parity}, independently of the proof in Appendix~\ref{app:computation}, and then checks each displayed line of that proof.
\texttt{derivative\_family\_sym.py} reproduces the two values of $\zetaR$ in Example~\ref{ex:derivative_family} exactly, and \texttt{bouss\_sym.py} reproduces the closed form~\eqref{eq:boussinesq_zeta} symbolically.
\texttt{bouss\_num.py} evaluates~\eqref{eq:boussinesq_zeta} numerically by a direct cascade computation and by the matrix resolvents of \texttt{o2sym}.
The two agree to relative $3\times10^{-15}$ over $200$ random admissible parameter sets including $\beta=0$, with $\xiR=0$ to machine precision in every run.
Only \texttt{bouss\_num.py} imports \texttt{o2sym}.
The other three need \texttt{sympy} alone.

\appendix

\section{Proof of Proposition~\ref{prop:ham_explicit_defects}}\label{app:computation}

This appendix proves Proposition~\ref{prop:ham_explicit_defects} in three steps: reduction of the projection to an imaginary part, extraction of the two direct vectors, and the cascade through the second harmonic.
Throughout, $G_2$ and $G_3$ are derivative-free for a common multiplier $\mathcal K=\partial_xP(-\partial_x^2)$, so the flux is~\eqref{eq:general_equivariant_flux} with real coefficients and no variational hypothesis is imposed.
Abbreviate $\ksym_1=\ksym(m_c)$, $\ksym_2=\ksym(2m_c)$ and $\theta:=u_c\overline{v_c}$, so that $r_c=\Re\theta$ and $\chi=\Im\theta$ by~\eqref{eq:qc}.

\subsection{Reduction of the projection to an imaginary part}

Every coefficient below has the same form: a vector $W\in\mathbb C^2$ carried by $e^{im_cx}$, produced by the flux, then acted on by $\mathcal K$ and projected onto the first critical adjoint eigenvector.
We define the complex-linear functional
\begin{equation}\label{eq:app_pairing}
\pair{W}:=W_1\overline{v_c}+W_2\overline{u_c}
=W\mathbin{\cdot}(S\overline{\mathbf q}),
\qquad W\in\mathbb C^2 .
\end{equation}

Applying~\eqref{eq:swap_cancellation} to the field $We^{im_cx}$ gives, for $W\in\mathbb C^2$,
\begin{equation}\label{eq:app_projection}
\mathcal P_1\bigl(\mathcal K(We^{im_cx})\bigr)
=\varkappa\pair{W},
\qquad
\varkappa:=\frac{i\ksym_1}{2r_c},
\qquad
\Re\,\mathcal P_1\bigl(\mathcal K(We^{im_cx})\bigr)
=-\frac{\ksym_1}{2r_c}\,\Im\pair{W},
\end{equation}
the second identity because $\ksym_1$ and $r_c$ are real.

\subsection{The direct coefficients}

We write $\mathbf a=\mathbf q$ and $\mathbf b=\mathbf q_B=-R_0\overline{\mathbf q}=(-\overline{u_c},\overline{v_c})^T$ for the vectors carried by $A$ and $B$ in~\eqref{eq:center_linear_parametrization}, so that $\overline A$ and $\overline B$ carry $\overline{\mathbf a}$ and $\overline{\mathbf b}=(-u_c,v_c)^T$.
By~\eqref{eq:direct_coefficient_extraction}, $\hat s_{1,1}$ and $\hat s_{1,2}$ are~\eqref{eq:app_projection} applied to the coefficients $W^{\mathrm s}$ of $A^2\overline A$ and $W^{\mathrm m}$ of $AB\overline B$ in $F_3(\psi_c)$.
Direct extraction from $F_3$ on the legs of~\eqref{eq:center_linear_parametrization}, followed by~\eqref{eq:app_pairing}, gives
\[
\begin{aligned}
\pair{W^{\mathrm s}}
&=h_{11}\bigl(\theta^2+2\abs\theta^2\bigr)
+h_{22}\bigl(\overline\theta^2+2\abs\theta^2\bigr)
+3h_{12}\abs{v_c}^4+3h_{21}\abs{u_c}^4,\\
\pair{W^{\mathrm m}}
&=-2h_{11}\theta^2-2h_{22}\overline\theta^2
+6h_{12}\abs{v_c}^4+6h_{21}\abs{u_c}^4.
\end{aligned}
\]
Taking imaginary parts gives
\begin{equation}\label{eq:app_imaginary_parts_cubic}
\Im\pair{W^{\mathrm s}}=(h_{11}-h_{22})\,2r_c\chi,
\qquad
\Im\pair{W^{\mathrm m}}=-2(h_{11}-h_{22})\,2r_c\chi .
\end{equation}
Substituting into~\eqref{eq:app_projection} and recalling $\mathcal D_3=h_{11}-h_{22}$ proves~\eqref{eq:ham_cubic_self_identity}--\eqref{eq:ham_cubic_identity}.

\subsection{The cascade coefficient}

Let $\mathcal B$ denote the symmetric bilinear map associated with $F_2$, so that $F_2(U)=\mathcal B(U,U)$ and, for spatial fields $X,Y$, $\tfrac12D^2G(0)(X,Y)=\mathcal K\bigl(\mathcal B(X,Y)\bigr)$.
From~\eqref{eq:general_equivariant_flux},
\begin{equation}\label{eq:app_bilinear}
\mathcal B(\mathbf x,\mathbf y)
=\begin{pmatrix}
\tfrac12g_{11}(x_1y_2+x_2y_1)\\[2pt]
\tfrac12(g_{21}x_1y_1+g_{22}x_2y_2)
\end{pmatrix}.
\end{equation}

\emph{Step 1: only the second harmonic contributes.} By~\eqref{eq:cubic_reduced_field} the cascade term is $D^2G(0)(\psi_c,h_2)$, so $AB\overline B$ arises by pairing one linear leg of $\psi_c$ with one quadratic monomial of $h_2$, in the three splittings $A\times B\overline B$, $B\times A\overline B$ and $\overline B\times AB$.
The first two involve $B\overline B$ and $A\overline B$, of output mode $0$ by~\eqref{eq:quadratic_mode_frequency_table}, so $\Phi_{B\overline B}=\Phi_{A\overline B}=0$ on the mean-zero phase space by~\eqref{eq:zero_mode_vanishing}.
Only the third splitting survives:
\begin{equation}\label{eq:app_cascade_split}
[AB\overline B]\bigl(DG_2(\psi_c)h_2\bigr)
=2\,\mathcal K\Bigl(\mathcal B\bigl(\overline{\mathbf b},\Phi_{AB}\bigr)e^{im_cx}\Bigr),
\end{equation}
the output mode being $-m_c+2m_c=m_c$.
Thus only the unshifted $AB$ resolvent enters.

\emph{Step 2: the second-harmonic coefficient.} By~\eqref{eq:quadratic_forcing_definition} the $AB$ forcing is $\hat g_{AB}=2i\ksym_2\,\mathcal B(\mathbf a,\mathbf b)$, the factor $2$ arising from the two mixed terms in the bilinear expansion and $i\ksym_2$ being the symbol of $\mathcal K$ at the output mode $2m_c$.
From~\eqref{eq:app_bilinear} and~\eqref{eq:center_linear_parametrization},
\begin{equation}\label{eq:app_V}
\mathcal B(\mathbf a,\mathbf b)
=\begin{pmatrix}
\tfrac12g_{11}\bigl(u_c\overline{v_c}-\overline{u_c}v_c\bigr)\\[2pt]
\tfrac12\bigl(g_{22}\abs{v_c}^2-g_{21}\abs{u_c}^2\bigr)
\end{pmatrix}
=\begin{pmatrix}i\Gamma\\ \Delta\end{pmatrix}, \qquad
\begin{aligned}
\Gamma&:=g_{11}\chi,\\
\Delta&:=\tfrac12\bigl(g_{22}\abs{v_c}^2-g_{21}\abs{u_c}^2\bigr),
\end{aligned}
\end{equation}
with $\Gamma,\Delta\in\mathbb R$.
Since $\beta_{AB}=0$, \eqref{eq:mixed_second_harmonic} and~\eqref{eq:inverse_symbol_parity} give
\begin{equation}\label{eq:app_Phi_explicit}
\Phi_{AB}
=-M_{2m_c}^{-1}\hat g_{AB}
=-2i\ksym_2
\begin{pmatrix}\varrho_{11}&i\varrho_{12}\\ i\varrho_{21}&\varrho_{22}\end{pmatrix}
\begin{pmatrix}i\Gamma\\ \Delta\end{pmatrix}
=\begin{pmatrix}\varphi_1\\ i\varphi_2\end{pmatrix},
\end{equation}
where
\begin{equation}\label{eq:app_phi}
\varphi_1=2\ksym_2\bigl(\varrho_{11}\Gamma+\varrho_{12}\Delta\bigr)\in\mathbb R,
\qquad
\varphi_2=-2\ksym_2\bigl(\varrho_{22}\Delta-\varrho_{21}\Gamma\bigr)\in\mathbb R.
\end{equation}
In particular, $\overline{\Phi_{AB}}=R_0\Phi_{AB}$, which is~\eqref{eq:vertex_Psi_conjugation}.

\emph{Step 3: the outer interaction.}
By~\eqref{eq:app_cascade_split} and~\eqref{eq:app_projection},
$\Re\hat c_{1,2}=-\frac{\ksym_1}{2r_c}\Im\pair{W^{\mathrm c}}$ with
$W^{\mathrm c}=2\mathcal B(\overline{\mathbf b},\Phi_{AB})$.
From~\eqref{eq:app_bilinear}, \eqref{eq:center_linear_parametrization}
and~\eqref{eq:app_Phi_explicit},
\[
W^{\mathrm c}
=\begin{pmatrix}
g_{11}\bigl(v_c\varphi_1-iu_c\varphi_2\bigr)\\[2pt]
-g_{21}u_c\varphi_1+ig_{22}v_c\varphi_2
\end{pmatrix},
\]
and therefore
\begin{equation}\label{eq:app_pair_cascade}
\pair{W^{\mathrm c}}
=\varphi_1\bigl(g_{11}\abs{v_c}^2-g_{21}\abs{u_c}^2\bigr)
+i\varphi_2\bigl(g_{22}\overline\theta-g_{11}\theta\bigr).
\end{equation}
The first summand is real, while $\theta=r_c+i\chi$ gives
\[
g_{22}\overline\theta-g_{11}\theta
=(g_{22}-g_{11})r_c-i(g_{11}+g_{22})\chi,
\]
so that
\begin{equation}\label{eq:app_imaginary_part_cascade}
\Im\pair{W^{\mathrm c}}
=\varphi_2(g_{22}-g_{11})r_c
=\varphi_2\,\mathcal D_2\,r_c .
\end{equation}

Substituting~\eqref{eq:app_phi} into~\eqref{eq:app_imaginary_part_cascade} and then into~\eqref{eq:app_projection}, and inserting $\Gamma$ and $\Delta$ from~\eqref{eq:app_V}, gives
\[
\Re\hat c_{1,2}
=-\frac{\ksym_1}{2r_c}\varphi_2\mathcal D_2r_c
=\ksym_1\ksym_2\,\mathcal D_2\bigl(\varrho_{22}\Delta-\varrho_{21}\Gamma\bigr)
=\frac{\ksym_1\ksym_2}{2}\,\mathcal D_2\,E,
\]
with $E$ as in~\eqref{eq:E_of_L}, which is~\eqref{eq:ham_quadratic_identity}.

\section{Proof of Lemma~\ref{lem:witness}}\label{app:converse}

\begin{proof}[Proof of Lemma~\ref{lem:witness}]
The symbol is $M_k(\lambda)=\begin{psmallmatrix}t+a&irk\\ ir^{-1}k&t-a\end{psmallmatrix}$ with $t=\lambda-\nu(k^2-m_c^2)^N$, so $\tr M_k(\lambda)=2t$ and $\det M_k(\lambda)=t^2-a^2+k^2$.
For Assumption~\ref{ass:H}, at $\lambda=0$, $\tr M_{m_c}=0$ and $\partial_\lambda\tr M_{m_c}=2$.
Moreover, for every $k\ne0$, $\det M_k(0)=\nu^2(k^2-m_c^2)^{2N}+k^2-a^2>0$ because $k^2\ge1>a^2$.
This is where the discreteness of the spectrum is used.
These facts prove~\ref{H1} and~\ref{H2}.
Since $N$ is even, $\tr M_k(0)=-2\nu(k^2-m_c^2)^N<0$ for $k\ne\pm m_c$, which is~\ref{H3}.
Finally, $\det M_k(0)\sim\nu^2k^{4N}$ and $\abs{\tr M_k(0)}\sim2\nu k^{2N}$ give~\ref{H4}.
For Assumption~\ref{ass:D}, conjugation by $R_0$ gives~\ref{E}.
Evenness of $N$ makes the principal part $-\nu\partial_x^{2N}$ on both diagonal entries, while the off-diagonal terms have order $1\le2N-1$ and $\lambda$ enters at order zero, which give~\ref{D} and~\ref{P} with $m_L=N$.
The inequality $2N-1\ge r_G$ is~\ref{S4}.

At $k=m_c$ and $\lambda=0$ the symbol is $\begin{psmallmatrix}a&irm_c\\ ir^{-1}m_c&-a\end{psmallmatrix}$, with $\omega_c=\sqrt{m_c^2-a^2}>0$ and, normalizing $v_c=1$, second row $ir^{-1}m_cu_c-a=i\omega_c$, so $u_c=r(\omega_c-ia)/m_c$ and
\[
\chi=\Im(u_c\overline{v_c})=-\frac{ar}{m_c},
\qquad
\abs{u_c}^2=\frac{r^2(\omega_c^2+a^2)}{m_c^2}=r^2 .
\]
At $k=2m_c$, with $d=\nu(3m_c^2)^N$,
\[
M_{2m_c}=\begin{pmatrix}-d+a&2irm_c\\ 2ir^{-1}m_c&-d-a\end{pmatrix},
\qquad
\det M_{2m_c}=d^2-a^2+4m_c^2>0,
\]
and inverting gives the entries $\varrho_{21}$, $\varrho_{22}$ of~\eqref{eq:witness_invariants} in the notation of~\eqref{eq:inverse_symbol_parity}.
\end{proof}

\section{Proof of Proposition~\ref{prop:boussinesq_ray}}\label{app:boussinesq}

\begin{proof}[Proof of Proposition~\ref{prop:boussinesq_ray}]
Write $c_\beta^2=c_0^2+\beta$ and normalize the critical eigenvector by its second component.
At onset the symbol of Example~\ref{ex:boussinesq} and~\eqref{eq:critical_frequency} give, writing $\omega=\omega_c$,
\[
M_m=\begin{pmatrix}-am^4&im\\ im(c_0^2+\beta m^2)&\epsilon m^2-\alpha_4m^4\end{pmatrix},
\qquad
\tr M_1=0,
\qquad
\omega^2=\det M_1=c_\beta^2-a^2,
\]
and, in the notation of~\eqref{eq:qc},
\[
\mathbf q=(u_c,1)^T,
\qquad
u_c=\frac i{i\omega+a},
\qquad
\abs{u_c}^2=\frac1{c_\beta^2},
\qquad
r_c=\Re u_c=\frac\omega{c_\beta^2}.
\]
Here $G_2(U)=\bigl(0,-\tfrac\alpha2\partial_x(v^2)\bigr)^T$ for $U=(v,p)$, so~\eqref{eq:quadratic_forcing_definition} and~\eqref{eq:self_second_harmonic} give
\[
\hat g_{AA}=\begin{pmatrix}0\\-i\alpha u_c^2\end{pmatrix},
\qquad
\Phi_{AA,1}=\frac{2\alpha u_c^2}\Theta,
\qquad
\Theta:=\det(2i\omega I-M_2),
\]
the second because $\hat g_{AA}$ has no first component and the $(1,2)$ entry of $(2i\omega I-M_2)^{-1}$ is $(M_2)_{12}/\Theta=2i/\Theta$.
In~\eqref{eq:cascade_coefficient_extraction} only the pairing of $\overline A\,\overline{\mathbf q}e^{-ix}$ with $A^2\Phi_{AA}e^{2ix}$ survives, since $\Phi_{A\overline A}=0$ by~\eqref{eq:zero_mode_vanishing}, so with the adjoint~\eqref{eq:ham_adjoint} and $S\overline{\mathbf q}=(1,\overline{u_c})^T$,
\[
\zeta=\hat c_{1,1}
=\frac{(S\overline{\mathbf q})^T}{2r_c}\begin{pmatrix}0\\-i\alpha\overline{u_c}\,\Phi_{AA,1}\end{pmatrix}
=-\frac{i\alpha\overline{u_c}^{\,2}\Phi_{AA,1}}{2r_c}
=-\frac{i\alpha^2\abs{u_c}^4}{r_c\Theta}
=-\frac{i\alpha^2}{c_\beta^2\,\omega\,\Theta}.
\]
Expanding $\Theta$ and eliminating $\omega^2=c_\beta^2-a^2$,
\[
\Theta=12\beta+12a(16\alpha_4-5a)+24i\omega\epsilon,
\qquad
\zetaR=\Re\zeta=-\frac{\alpha^2\,\Im\Theta}{c_\beta^2\,\omega\,\abs\Theta^2}
=-\frac{24\alpha^2\epsilon}{c_\beta^2\abs\Theta^2},
\]
and $\abs\Theta^2=144\bigl([\beta+a(16\alpha_4-5a)]^2+4\epsilon^2\omega^2\bigr)$ gives the closed form
\begin{equation}\label{eq:boussinesq_zeta}
\zetaR=-\frac{\alpha^2\,\epsilon}{6c_\beta^2\bigl([\beta+a(16\alpha_4-5a)]^2+4\epsilon^2(c_\beta^2-a^2)\bigr)},
\end{equation}
which is negative for every such $(a,\alpha_4)$.
Since $a(16\alpha_4-5a)=O(\epsilon^2)$, the bracket is $\beta^2+O(\epsilon^2)$ when $\beta>0$ and $4c_0^2\epsilon^2\bigl(1+O(\epsilon^2)\bigr)$ when $\beta=0$, which gives the two expansions with $C_\beta=\alpha^2/(6c_\beta^2\beta^2)$ and $C_0=\alpha^2/(24c_0^4)$.
At $\epsilon=0$ the block $M_m$ has eigenvalues $\pm i\omega_m$ with $\omega_m^2=\det M_m=m^2\bigl(c_\beta^2+\beta(m^2-1)\bigr)$, so the detuning $\delta=2\omega_1-\omega_2$ vanishes exactly when $\beta=0$, which is the detuning claim of~\ref{B1} and~\ref{B2}.
\end{proof}

\bibliographystyle{abbrv}
\bibliography{hamO2}

\end{document}